\documentclass[a4paper]{article}

\usepackage{amsfonts,amsmath,amsthm}
\usepackage{amsthm}
\usepackage{a4wide}
\usepackage{epsfig,graphicx,color}

\numberwithin{equation}{section}

\newtheorem{theo}{\bf Theorem}[section]
\newtheorem{lem}[theo]{Lemma}

\newtheorem{proposition}[theo]{Proposition}
\newtheorem{lemma}[theo]{Lemma}

\newtheorem{theorem}[theo]{Theorem}

\newtheorem{definition}[theo]{Definition}

\theoremstyle{remark}
\newtheorem{rem}[theo]{Remark}
\newtheorem{remark}[theo]{Remark}

\newcommand{\R}{\mathbb{R}}

\newcommand{\eps}{\varepsilon}

\newcommand{\dz}{{\mathcal{D}}}
\newcommand{\dzo}{\mathcal{D}_0}
\newcommand{\dzun}{\mathcal{D}_{\mathrm{junction}}}
\newcommand{\dzd}{\mathcal{D}_{\mathrm{linear}}}
\newcommand{\dzt}{\mathcal{D}_{\mathrm{implicit}}}
\newcommand{\dzq}{\mathcal{D}_{\mathrm{straight}}}

\begin{document}

\title{A Hamilton-Jacobi approach to junction problems 
and application to traffic flows}

\author{C. Imbert\footnote{Universit\'e Paris-Dauphine, CEREMADE, UMR
  CNRS 7534, place de Lattre de Tassigny, 75775 Paris Cx 16 \& Ecole
  Normale Sup\'erieure, D\'epartement de Math\'ematiques et
  Applications, UMR 8553, 45 rue d'Ulm, F 75230 Paris cedex 5, France,
  \texttt{imbert@ceremade.dauphine.fr}}, 
R. Monneau\footnote{Universit\'e Paris-Est, Ecole des Ponts
  ParisTech, CERMICS, 6 et 8 avenue Blaise Pascal, Cit\'e Descartes
  Champs-sur-Marne, 77455 Marne-La-Vall\'ee Cedex 2, France
  \texttt{monneau@cermics.enpc.fr}}, H. Zidani\footnote{ENSTA
  ParisTech\& INRIA Saclay (Commands INRIA team), 32 boulevard Victor,
  75379 Paris Cedex 15, France, \texttt{Hasnaa.Zidani@ensta.fr}}}

\date{}

\maketitle

\begin{center}
\textit{Dedicated to J.-B. Hiriart-Urruty}
\end{center}

\begin{abstract}
  This paper is concerned with the study of a model case of first order
  Hamilton-Jacobi equations posed on a ``junction'', that is to say the
  union of a finite number of half-lines with a unique common point. The
  main result is a comparison principle. We also prove existence and
  stability of solutions. The two challenging difficulties
  are the singular geometry of the domain and the discontinuity of the
  Hamiltonian.  As far as discontinuous Hamiltonians are concerned, these
  results seem to be new.  They are applied to the study of some models
  arising in traffic flows.  The techniques developed in the present
  article provide new powerful tools for the analysis of such 
  problems.
\end{abstract}
\bigskip

 \medskip

\medskip

\section{Introduction}

In this paper we are interested in Hamilton-Jacobi (HJ) equations
posed on a one dimensional domain containing one single singularity.
This is a special case of a more general setting where HJ equations
are posed in domains that are unions of submanifolds whose dimensions
are different \cite{bh07}. An intermediate setting is the study of HJ
equations on networks, {see in particular
  \cite{acct}}. We will restrict ourselves to a very simple network:
the union of a finite numbers of half-lines of the plane with a single
common point.  Such a domain is referred to as a \emph{junction} and
the common point is called the \emph{junction
  point}. {We point out that getting a comparison
  principle is the most difficult part in such a study; it is obtained
  in \cite{acct} for similar special networks (bounded star-shaped
  ones).}  Our motivation comes from traffic flows. For this reason,
it is natural to impose different dynamics on each \emph{branch} of
the junction. Consequently, the resulting Hamiltonian is by nature
discontinuous at the junction point. Together with the singularity of
the domain, this is the major technical difficulty to overcome. The
analysis relies on the complete study of some \emph{minimal action}
(or \emph{metric}) related to the optimal control interpretation of
the equation \cite{siconolfi03,fathi97}. We prove in particular that
this minimal action is semi-concave by computing it.

We first present the problem and the main results in details. Then we
recall existing results and compare them with ours.

\subsection{Setting of the problem}
\label{ss13}

In this subsection, the analytical problem is introduced in
details. We first define the junction, then the space of functions on
the junction and finally the Hamilton-Jacobi equation.

\paragraph{The junction.} Let us consider $N\ge 1$ different unit
vectors $e_i\in \R^2$ for $i=1,...,N$. We define the branches
$$
J_i= [0,+\infty)\cdot e_i,\quad J_i^{*}=J_i\setminus
\left\{0\right\},\quad i=1,...,N
$$
and the {\it junction} (see Figure \ref{f6})
$$
J=\bigcup_{i=1,...,N} J_i.
$$
The origin $x=0$ is called the {\it junction point}.  It is also
useful to write $J^*=J \setminus \{0\}$.
\begin{figure}[ht]
\centering\resizebox{3cm}{!}{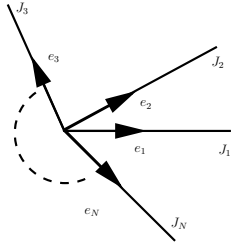}
\caption{A junction \label{f6}}
\end{figure}
For a time $T>0$, we also define
$$
J_T = (0,T)\times J.
$$
The reader can remark that we chose to embed the junction in a
two-dimensional Euclidean space.  But we could also have considered an
abstract junction, or we could have embedded it for instance in a higher
dimensional Euclidean space. We made such a choice for the sake of
clarity. 

\paragraph{Space of functions.} For a function $u\colon J_T\to \R$, we
denote by $u^i$ the restriction of $u$ to $(0,T)\times J_i$.  Then we
define the natural space of functions on the junction
$$
C^1_*(J_T)=\left\{u\in C(J_T),\quad u^i \in
  C^1((0,T)\times J_i)\quad \mbox{for}\quad i=1,...,N\right\}.
$$
  In particular for
$u\in C^1_*(J_T)$ and $x=x_ie_i$ with $x_i\ge 0$, we define
$$
u_t(t,x)=\frac{\partial u^i}{\partial t}(t,x_ie_i) \quad \text{ and }
\quad u^i_x(t,x)= \frac{\partial u^i}{\partial x_i}(t,x_ie_i).
$$
Then we set
$$
\left\{\begin{array}{ll}
    u_x(t,x)= u^i_x(t,x) &\quad \mbox{if}\quad x\not=0,\\
    u_x(t,0) = (u^j_x(t,0))_{j=1,...,N} &\quad \mbox{if}\quad x=0.
\end{array}\right.
$$

\paragraph{HJ equation on the junction.} We are interested in
continuous functions $u\colon [0,T)\times J \to \R$ which are {\it
  viscosity solutions} (see Definition~\ref{defi:main}) on $J_T$ of
\begin{equation}
\label{eq:main}
u_t + H(x,u_x) = 0
\end{equation}
submitted to an initial condition
\begin{equation}\label{eq:ci}
u(0,x) = u_0 (x), \quad x \in J.
\end{equation}
Because of the optimal control problem we have in mind (see Subsection
\ref{ss13} below), we restrict ourselves to the simplest case of
discontinuous Hamiltonians; precisely, we consider
$$
  H(x,p) = \left\{\begin{array}{lll}
      H_i(p)& \quad \mbox{for}\quad p\in\R & \quad \mbox{if}\quad x\in J_i^*\\
       \max_{i=1,...,N} H^-_i(p_i) &\quad \mbox{for}\quad p=(p_1,...,p_N)\in\R^N & \quad \mbox{if}\quad x=0
\end{array}\right.
$$
where $H_i$ are convex functions whose Legendre-Fenchel transform is
denoted $L_i$. We recall that 
$$
H_i(p)=L_i^*(p)=\sup_{q\in\R}\left(pq -L_i(q)\right)
$$
and
\begin{equation}\label{eq::r14}
H_i=L_i^*.
\end{equation}
We also consider
\begin{equation}\label{eq::r13}
  H_i^-(p_i)=\sup_{q\le 0} \left(p_i q -L_i(q)\right)
\end{equation}
Therefore equation (\ref{eq:main}) can be rewritten as follows
\begin{equation}\label{eq::r12}
  \left\{\begin{array}{lll}
      u^i_t + H_i(u^i_x)=0 & \quad \mbox{on}\quad (0,T)\times J_i^* & \quad \mbox{for}\quad i=1,...,N,\\
       u_t +  \max_{i=1,...,N} H^-_i(u^i_x) = 0 & \quad \mbox{on}\quad (0,T)\times \left\{0\right\}. &
\end{array}\right.
\end{equation}
\paragraph{The optimal control framework.}
In this paragraph, we give an optimal control interpretation
\cite{lions82,barles94,bcd} of the Hamilton-Jacobi equation.  We define the
set of admissible controls at a point $x\in J$ by
$$
A(x)=\left\{\begin{array}{ll}
    \R e_{i_0} & \quad \mbox{if}\quad x\in J_{i_0}^*,\\
       \bigcup_{i=1,...,N} \R^+ e_i & \quad \mbox{if}\quad
    x=0.
\end{array}\right. 
$$
For $(s,y), (t,x)\in [0,T]\times J$ with $s\le t$ (the case $s=t$
being trivial and forcing $y=x$), we define the set of admissible
trajectories from $(s,y)$ to $(t,x)$ by
\begin{equation}\label{eq::r16}
{\mathcal A}(s,y;t,x)=\left\{
X\in W^{1,1}([s,t];\R^2)\colon  \left|\begin{array}{ll}
X(\tau)\in J&  \mbox{for all}\quad \tau \in (s,t),\\
\dot{X}(\tau) \in A(X(\tau)) & \mbox{for a.e.}\quad \tau \in (s,t)\\
X(s)=y  & \mbox{and}\quad X(t)=x 
\end{array}\right.
\right\}.
\end{equation}
For $P=pe_i\in A(x)$ with $p\in \R$, we define the Lagrangian on the
junction
\begin{equation}\label{eq::r15}
L(x,P)=\left\{\begin{array}{ll}
L_{i}(p) &\quad \mbox{if}\quad x\in J_i^*\\
L_0(p) &\quad \mbox{if}\quad x=0
\end{array}\right. 
\end{equation}
with
$$
 L_0(p)=\min_{j=1,...,N}L_j(p).
$$
The reader can be surprised by the definition of $L(x,P)$ for $x=0$.
In fact, if one considers only trajectories that do not stay at the
junction point, then the ones staying at $0$ are approximated by those
staying very close to $0$ on a branch $i \in I_0$ and moving
``slowly'' ($\dot{X} \simeq 0$).

\subsection{Main results}

We make the following assumptions: 

\paragraph{\bf (A0)} The initial data $u_0$ is Lipschitz continuous. 

\paragraph{\bf (A1)} There exists a constant $\gamma>0$, and for all
$i=1,...,N$, there exists $C^2(\R)$ functions $L_i$ satisfying
$L_i''\ge \gamma >0$, such that (\ref{eq::r14}) and (\ref{eq::r13})
hold.
\begin{theo}[Existence and uniqueness]\label{th::r2}  
  Assume (A0)-(A1) and let $T>0$. Then there exists a unique
  viscosity solution $u$ of (\ref{eq:main})-(\ref{eq:ci}) on $J_T$ in
  the sense of Definition~\ref{defi:main}, satisfying for some
  constant $C_T>0$
$$
|u(t,x)-u_0(x)|\le C_T \quad \mbox{for all}\quad (t,x)\in J_T.
$$
Moreover the function $u$ is Lipschitz continuous with respect to
$(t,x)$ on $J_T$.
\end{theo}
On one hand, we will see below that the existence of a solution can be
obtained with Perron's method under weaker assumptions than (A1) (see
Theorem~\ref{thm:perron}). On the other hand, we are able to get
uniqueness of the solution only under assumption (A1) and this is a
consequence of the following result:
\begin{theo}[Comparison principle]\label{th::r1}
  Assume (A0)-(A1). Let $T>0$ and let $u$ (resp. $v$) be a sub-solution
  (resp. a super-solution) of (\ref{eq:main})-(\ref{eq:ci}) on $J_T$ in
  the sense of Definition~\ref{defi:main}.  We also assume that there
  exists a constant $C_T>0$ such that for all $(t,x)\in J_T$
$$
u(t,x) \le C_T (1+|x|)  \quad \left(\mbox{resp.}\quad v(t,x) \ge - C_T
  (1+|x|)\right).
$$
Then we have $u\le v$ on $J_T$.
\end{theo}
In order to prove this strong uniqueness result, we will use in an
essential way the \emph{value function} associated to the optimal control
problem described in Subsection~\ref{ss13}: for $t\ge 0$, 
\begin{equation}\label{eq::r17}
  u_{\mathrm{oc}}(t,x)=\inf_{y\in J,\ X\in {\mathcal A}(0,y;t,x)}
  \left\{u_0(y)+ \int_0^t L(X(\tau),\dot{X}(\tau))d\tau\right\}
\end{equation}
where $L$ is defined in (\ref{eq::r15}) and ${\mathcal A}(0,y;t,x)$ is
defined in (\ref{eq::r16}). 
\begin{theo}[Optimal control representation of the solution]\label{th::r3}
  Assume (A0)-(A1) and let $T>0$. The unique solution given by Theorem
  \ref{th::r2} is $u=u_{\mathrm{oc}}$ with $u_{\mathrm{oc}}$ given in (\ref{eq::r17}).
  Moreover, we have the following Hopf-Lax representation formula
\begin{equation}\label{eq::r77}
u_{\mathrm{oc}}(t,x)=\inf_{y\in J}\left\{u_0(y)+ {\mathcal
    D}\left(0,y; t,x \right)\right\}
\end{equation}
with
$$
{\mathcal D}(0,y;t,x)=\inf_{X\in {\mathcal A}(0,y;t,x)}\left\{\int_0^t
  L(X(\tau),\dot{X}(\tau))d\tau\right\}.
$$
\end{theo}
The comparison principle is obtained by combining 
\begin{itemize}
\item  a super-optimality principle for counterrevolutions $v$, which
  implies $v\ge u_{\mathrm{oc}}$;
\item a direct comparison result with sub solutions $u$, which gives
$u_{\mathrm{oc}}\ge u$.
\end{itemize}
We finally have the following result which shed light on the role of
the junction condition (see the second line of (\ref{eq::r12})).
\begin{theo}[Comparison with continuous solutions outside the junction
  point]\label{th::r4}
  Assume (A0)-(A1) and let $T>0$. Let $u\in C([0,T)\times J)$ be such
  that $u(0,\cdot)=u_0$ and for each $i\in \left\{1,...,N\right\}$,
  the restriction $u^i$ of $u$ to $(0,T)\times J_i$ is a classical
  viscosity solution of
$$
u^i_t + H_i(u^i_x)=0 \quad \mbox{on}\quad (0,T)\times J_i^*.
$$ 
Then $u$ is a sub-solution of (\ref{eq:main})-(\ref{eq:ci}) on $J_T$ in
the sense of Definition~\ref{defi:main}, and $u\le u_{\mathrm{oc}}$.
\end{theo}
An immediate consequence of Theorem~\ref{th::r4} is the fact that the
optimal control solution $u_{\mathrm{oc}}$ is the maximal continuous function
which is a viscosity solution on each open branch.

We apply in Section~\ref{s5} our HJ approach to describe traffic flows
on a junction. In particular, we recover the well-known junction
conditions of Lebacque (see \cite{lebacque}) or equivalently those for the
Riemann solver at the junction as in the book of Garavello and Piccoli
\cite{gp}; see also \cite{gp09}.

\subsection{Comments}

We already mentioned that the main difficulties we have to overcome in
order to get our main results are on one hand the singular geometry of the
domain and on the other hand the discontinuity of the Hamiltonian. 

\medskip

\paragraph{\bf Discontinuity.}  \ 
{ Several papers in the literature deal with HJB equations
  with discontinuous coefficients; see for instance
  \cite{barles93,soravia93,ns95,soravia02,stromberg03,bd05,soravia06,cr07,ch08}.
  Note that in these works the optimal trajectories do cross the
  regions of discontinuities (i.e. the junction in the present paper)
  only on a set of time of measure zero. In the present paper, the
  optimal trajectories can remain on the junction during some time
  intervals, and the results cited above do not apply then to the
  problem studied here.

On the other hand, the analysis of scalar conservation laws with
discontinuous flux functions has been extensively studied, we refer to
\cite{sv,bv06,bk08} and references therein. We also point out that a
uniqueness result is proved in $\R$ in the framework corresponding a
junction with two branches \cite{gnpt07}. To the best of our knowledge,
in the case of junctions with more than two branches, there are no
uniqueness result. Moreover, the link between HJB
equations and conservation laws with discontinuous has been seldom
investigated \cite{Ostrov}.}

{The main differences between the study in \cite{acct}
  and the one carried out in the present paper lie in the fact that in
  \cite{acct} the Lagrangian can depend on $x$ and is continuous with
  respect to this variable, while we consider a Lagrangian which is
  constant in $x$ on each branch but can be discontinuous (with
  respect to $x$) at the junction. We point out that we cannot extend
  directly our approach to Lagrangians depending on $x$ since we use
  extensively the representation formula ``\`a la Hopf-Lax''. In order
  to generalise results in this direction, the semi-concavity of the
  ``fundamental solution'' $\mathcal{D}$ should be proved without
  relying on explicit computations. This question is very interesting
  but is out of the scope of the present paper.}
\medskip

\paragraph{\bf Networks.}  \ 
{ It is by now well known that the study of traffic flows
  on networks is an important motivation that give rise to several
  difficulties related to scalar conservation laws with discontinuous
  coefficients.  This topic has been widely studied by many authors,
  see for instance \cite{bk08, gp,ekfn} and the references therein.

However, the study of HJB equations on networks has been considered
very recently; the reader is referred to \cite{schieborn,cs} where
Eikonal equations are considered.  A more general framework was also
studied in \cite{acct,acct-2} where a definition of viscosity
solutions on networks, similar to Definition~\ref{defi:main}, has been
introduced.}

{It would be interesting to extend the results of the
  present paper to more general networks but the obstacle is the same
  than the one to be overcome if one wants to deal with Lagrangians
  depending on $x$: for a general network, the complete study of the
  fundamental solution is probably out of reach. This is the reason
  why we only consider the very specific case of a junction in order
  to be able to overcome the difficulty of the discontinuity of the
  Lagrangian.}

\medskip
\paragraph{\bf The optimal control interpretation.} 
\ { As explained above, the comparison principle is proved
  by using in an essential way the optimal control interpretation of
  the Hamilton-Jacobi equation. The use of representation formulas
  and/or optimality principles is classical in the study of
  Hamilton-Jacobi equations
  \cite{ls85,soravia99a,soravia99b,gs04,gs06}.  More specifically, it
  is also known that a ``metric'' interpretation of the
  Hamilton-Jacobi equation is fruitful \cite{siconolfi03}. Such an
  interpretation plays a central role in the weak KAM theory
  \cite{fathi97}.

As far as our problem is concerned, we are not able to prove
uniqueness of viscosity solutions by using the classical techniques of
doubling variable.  The idea used here is based on the equivalence
between the viscosity super-solution and the super-optimality
principle (also known as weak-invariance principle), and by using
representation formulas for the viscosity sub solutions.  This
representation seems to be new for HJB equations with discontinuous
coefficients, see for instance \cite{DalMFra}.  }

We would like next to be a bit more precise. The technical core of the
paper lies in Theorem~\ref{thm:minimal-action}. This result implies that
the function 
$$
\mathcal{D}(s,y;t,x)=(t-s) \mathcal{D}_0 \left(\frac{y}{t-s},\frac{x}{t-s}\right)
$$ 
is semi-concave with respect to $(t,x)$ and
$(s,y)$ and, if there are at least two branches ($N\ge 2$), that
$\mathcal{D}$ satisfies
\begin{equation*}\label{eq::r150}
\left\{\begin{array}{rl}
\dz_t + H(x, \dz_x)&=0,\\
-\dz_s+ H(y, -\dz_y)&=0
\end{array}\right.
\end{equation*}
(in a weak sense made precise in the statement of
Theorem~\ref{thm:minimal-action}).  In the case where the Lagrangians
coincide at the junction point ($L_1(0)=...=L_N(0)$), it turns out
that the restriction $\dzo^{ji}(y,x)$ of $\dzo$ to $J_j\times J_i$
belong to $C^1(J_j\times J_i)$ and is convex. A more general case is
considered in this paper: Lagrangians can differ at the junction point
and in this case, the functions $\dzo^{ji}$ are not convex nor $C^1$
anymore for some $(i,j)$.  { Let us point out here that the
  assumptions on the Hamiltonian $H_i$, and in particular the fact
  that it does not depend on the space variable $x$, plays a crucial
  role to establish the properties satisfied by the minimal action
  function ${\cal D}$.}

\paragraph{\bf Generalization and open problems.}
Eventually, we briefly mention natural generalizations of our results
and some important open problems. First of all, it { would be natural
  to extend the results of this paper to more general setting where
  the Hamiltonians $H_i$ depend on the space variable $x$.  Moreover,
  it would be interesting to consider general networks with several
  junction points.  Dealing with non-convex and non-coercive
  Hamiltonians is quite challenging and would require first to have a
  direct proof of the comparison principle which does not need to go
  through the interpretation of the viscosity solution as the value
  function of an optimal control problem.}

\subsection*{Organization of the article}

Section~\ref{s5} is devoted to the application of our results to some
traffic flow problems. In particular, the HJ equation is derived and the
junction condition is interpreted.  In Section~\ref{sec:visc}, the
definition of (viscosity) solutions is made precise.  In
Section~\ref{sec:oc}, the first important properties of optimal
trajectories are given. Section~\ref{sec:op} is devoted to the proof of the
main results of the paper. In particular, the comparison principle is
proved by proving a super-optimality principle and by comparing
sub solutions with the solution given by the optimal control interpretation
of the equation. Section~\ref{sec:complete} is devoted to the proof of the
technical core of the paper, namely the existence of test functions for the
minimal action associated with the optimal control interpretation.

\subsection*{Notation} 

\paragraph{Distance and coordinates in the junction.}
We denote by $d$ the geodesic distance defined on $J$ by
$$
d(x,y)=\left\{\begin{array}{ll}
|x-y| &\quad \mbox{if}\quad x,y\quad \mbox{belong to the same branch $J_i$ for some $i$},\\
|x|+|y| &\quad \mbox{if}\quad x,y\quad \mbox{do not belong to the same branch}.
\end{array}\right.
$$
For $x \in J$, $B(x,r)$ denotes the (open) ball centered at $x$ of
radius $r$. We also consider balls $B((t,x),r)$ centered at $(t,x) \in
(0,+\infty) \times J$ of radius $r>0$.  For $x\in J$, let us define
the index $i(x)$ of the branch where $x$ lies. Precisely we set:
$$
i(x)=\left\{\begin{array}{ll}
i_0 & \quad \mbox{if}\quad x\in J_{i_0}^*,\\
0 & \quad \mbox{if}\quad x=0.
\end{array}\right.
$$
Up to reordering the indices, we assume that there exists an index
$k_0\in \left\{1,...,N\right\}$ such that
\begin{equation}\label{eq::r26}
L_0(0)=L_1 (0) = \dots = L_{k_0} (0) < L_{k_0+1} (0) \le \dots \le L_N (0).
\end{equation}
We also set
$$
  I_0= \left\{1,..,k_0 \right\} \quad \mbox{and}\quad I_N=\left\{ 1,...,N \right\}.
$$

\paragraph{Functions defined in $\mathbf{J^2}$.} For a function $\varphi$ defined on $J\times J$,
we call $\varphi^{ij}$ its restriction to $J_i\times J_j$.  Then we
define the space
$$
C^1_*(J^2)=\left\{\varphi\in C(J^2),\quad \varphi^{ij}\in
  C^1(J_i\times J_j)\quad \mbox{for all}\quad i,j\in I_N\right\}.
$$
We also call for $x=x_ie_i$ with $x_i\ge 0$ and $y=y_je_j$ with
$y_j\ge 0$
$$
\partial_x^i\varphi(x,y)=\frac{\partial}{\partial x_i}\varphi^{ij}(x_ie_i,y)\quad \mbox{and}\quad 
\partial_y^j\varphi(x,y)=\frac{\partial}{\partial
  y_j}\varphi^{ij}(x,y_je_j)
$$ 
and
$$
\partial_x \varphi(x,y)=\left\{\begin{array}{ll}
\partial_x^i\varphi(x,y)&\quad \mbox{if}\quad x\in J_i^*,\\
\left(\partial_x^i\varphi(x,y)\right)_{i=1,...,N} &\quad \mbox{if}\quad x=0
\end{array}\right.
$$
and similarly
$$
\partial_y \varphi(x,y)=\left\{\begin{array}{ll}
\partial_y^j\varphi(x,y)&\quad \mbox{if}\quad y\in J_j^*,\\
\left(\partial_y^j\varphi(x,y)\right)_{j=1,...,N} &\quad \mbox{if}\quad y=0.
\end{array}\right.
$$
We also set
$$
\left\{\begin{array}{ll}
x\partial_x \varphi(x,y)= x_i \partial_x^i\varphi (x,y) &\quad \mbox{if}\quad x\in J_i,\\
y\partial_y \varphi(x,y)= y_j \partial_y^j\varphi (x,y) &\quad \mbox{if}\quad y\in J_j.
\end{array}\right.
$$

\section{Application to the modeling of traffic flows}\label{s5}

In this section we present the derivation of the Hamilton-Jacobi formulation of traffic on a junction.
We also discuss the meaning of our junction condition in this framework and relate it to known results.

\subsection{Primitive of the { densities} of cars}\label{s51}

We consider a junction (represented on Figure \ref{f5})
 \begin{figure}[h]
 \centering \resizebox{4cm}{!}{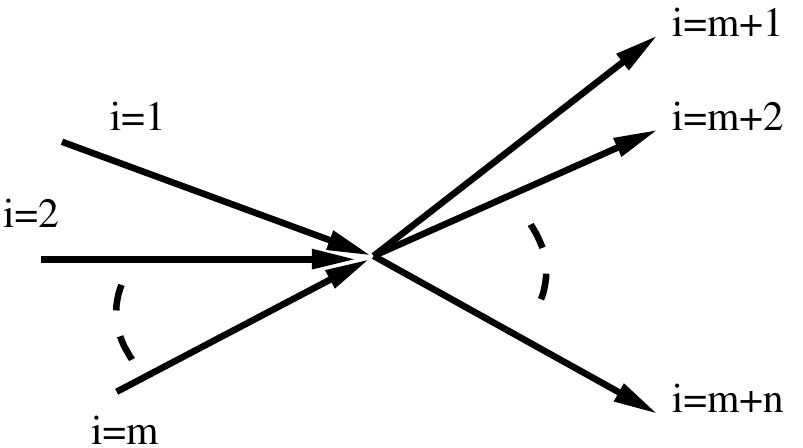}
 \caption{A traffic junction \label{f5}}
 \end{figure}
 with $m\ge 1$ incoming roads (labeled by the index $i=1, ...,m$) and $n\ge
 1$ outgoing roads (labeled by $j=m+1,...,m+n=N$). This means that the cars
 move on the incoming roads in the direction of the junction and then have
 to choose to go on one of the $n$ outgoing roads.  We assume that the
 proportion of cars coming from the branch $i=1,...,m$ is a fixed number
 $\gamma_i>0$ (which may be not realistic for $m\ge 2$), and that the
 proportion of cars going on each branch $j\in \left\{m+1,...,m+n\right\}$
 is also a fixed number $\gamma_j>0$.  We also assume the obvious relations
 (for conservation of cars)
$$\sum_{i=1,...,m} \gamma_i =1\quad \mbox{and}\quad \sum_{j=m+1,...,m+n} \gamma_j =1$$
We denote by $\rho^k(t,X)\ge 0$ the car density at time $t$ and at the
position $X$ on the branch $k$.  In particular, we assume that the traffic
is described on each branch $k$ by a flux function $f^k\colon\R\to \R$. We
assume
\begin{equation}\label{eq::r0}
\mbox{each function $f^k$ is concave and has a unique maximum value at $\rho=\rho^k_c>0$}
\end{equation}
The typical example of such flux function is given by the LWR model (Lighthill, Whitham \cite{lw} 
and Richards \cite{richards}),
with 
\begin{equation}\label{eq::r155}
  f(\rho)=\rho v(\rho)\quad \mbox{with the velocity}\quad v(\rho)=V_{\mathrm{max}}(1-\rho/\rho_{\mathrm{max}})
\end{equation}
where $V_{\mathrm{max}}$ and $\rho_{\mathrm{max}}$ are respectively
the maximal velocity and the maximal car density in the model). In this
model the critical car density $\rho_c$ where $f$ is maximal, is equal to
$\frac12 \rho_{\mathrm{max}}$.

We assume that the car densities are solution of non linear transport equations:
\begin{equation}\label{eq::r1}
\left\{\begin{array}{ll}
\rho^i_t + (f^i(\rho^i))_X =0, \quad X<0,& \quad \mbox{for}\quad i=1,...,m\\
\rho^j_t + (f^j(\rho^j))_X =0, \quad X>0,& \quad \mbox{for}\quad j=m+1,...,m+n
\end{array}\right.
\end{equation}
where we assume that the junction point is located at the origin $X=0$.

We do not precise yet the junction condition at $X=0$, and we now proceed
formally to deduce the Hamilton-Jacobi model of such a junction.  For a
function $g$ to determine, let us consider the functions

\begin{equation}\label{eq::r2}
\left\{\begin{array}{lll}
 U^i(t,X)=g(t) + \frac{1}{\gamma^i}\int_0^X \rho^i(t,Y)\ dY,&\quad X<0,&\quad \mbox{for}\quad i=1,...,m,\\
 U^j(t,X)=g(t) + \frac{1}{\gamma^j}\int_0^X  \rho^j(t,Y)\ dY,&\quad X>0,& \quad \mbox{for}\quad j=m+1,...,m+n.
\end{array}\right.
\end{equation}
Then we can compute formally for $j=m+1,...,m+n$
$$
\begin{array}{ll}
U^j_t & = g'(t) + \frac{1}{\gamma^j}\int_0^X  \rho^j_t(t,Y)\ dY\\
\\
& = g'(t) - \frac{1}{\gamma^j}\int_0^X  (f^j(\rho^j(t,Y)))_X\ dY\\
\\
& = - \frac{1}{\gamma^j}f^j(\rho^j(t,X)) + g'(t) + \frac{1}{\gamma^j}f^j(\rho^j(t,0^+))
\end{array}
$$
This shows that for $j=m+1,...,m+n$
\begin{equation}\label{eq::r3}
U^j_t + \frac{1}{\gamma^j}f^j(\gamma_jU^j_X)=h^j(t)
\end{equation}
where
$$
h^j(t):=g'(t) + \frac{1}{\gamma^j}f^j(\rho^j(t,0^+)).
$$
Remark that we can show similarly that (\ref{eq::r3}) is still true for the
index $j$ replaced by $i=1,...,m$ with
$$
h^i(t)=g'(t) + \frac{1}{\gamma^i}f^i(\rho^i(t,0^-)).
$$
In particular, this shows (at least when the quantities in (\ref{eq::r3})
are well defined) that we can choose $g$ such that the total flux $-g'(t)$
is given by
\begin{equation}\label{eq::r4}
-g'(t) =\sum_{i=1,...,m}f^i(\rho^i(t,0^-))
\end{equation}
and then we have
$$
h^k(t)\equiv 0 \quad \mbox{for}\quad i=1,...,m+n
$$
if and only if
\begin{equation}\label{eq::r11}
\left\{\begin{array}{l}
f^i(\rho^i(t,0-))=\gamma^i (-g'(t)) \quad \mbox{for}\quad i=1,...,m\\
f^j(\rho^j(t,0^+))=\gamma^j (-g'(t)) \quad \mbox{for}\quad j=m+1,...,m+n
\end{array}\right.
\end{equation}
which is exactly the expected condition which says that the proportion of
incoming cars going in the junction from the branch $i$ is $\gamma^i$ and
the proportion of cars getting out of the junction which choose to go on
the branch $j$ is $\gamma^j$.

Let us notice that if we choose the initial condition $g(0)=0$, then we
deduce from (\ref{eq::r4}) that we have with $l=1,...,m+n$
$$-U^l(t,0)=-g(t)=\int_0^t \sum_{i=1,...,m}f^i(\rho^i(\tau,0^-))\ d\tau$$
which shows that \emph{$-U^l(t,0)$ can be interpreted as the total quantity
  of cars passing through the junction point $X=0$ during the time interval
  $[0,t)$}. As a consequence,
the quantity $-U^l_t(t,0)$ can also be interpreted as the instantaneous
flux of cars passing through the junction point. \medskip

We now give a further interpretation of the problem in the special case
$m=1$.  In the special case $m=1$, imagine for a moment, that we come back
to a discrete description of the traffic, where each car of label $k$ has a
 position $x_k(t)$ with the ordering $x_k(t)<x_{k+1}(t)<0$. We can
be interested in the label $k$ of the car $x_k(t)\le 0$ which is the
closest to the junction point $X=0$. Let us call it $K(t)$. We can
normalize the initial data such that $K(0)=0$.  Then the quantity of cars
that have passed through the junction point $X=0$ during the time interval
$[0,t)$ is equal to $-K(t)$, which is the exact discrete analogue of the
continuous quantity $-U^1(t,0)$.

On the other hand the number of cars between the positions $a=x_{A}(t)$ and
$b=x_{B}(t)$ is obviously equal to $B-A$, and its continuous analogue on
the branch $i=m=1$ with $\gamma^1=1$, is $\int_a^b \rho^1(t,X)\ dX =
U^1(t,b)-U^1(t,a)$.  This shows that $U^1(t,X)$ can be interpreted as the
exact continuous analogue of the discrete labeling of the cars moving in
the traffic.

This interpretation is also meaningful on the ``exit'' branches,
i.e. for $j\in \left\{m+1,...,m+n\right\}$. Indeed, for such $j$'s,
$U^j(t,X)$ is the continuous analogue of the discrete label of the car
that have decided to choose the branch $j$ and which is at time $t$
close to the position $X>0$.

\subsection{Getting the Hamilton-Jacobi equations}

We now set
\begin{equation}\label{eq::r6}
\left\{\begin{array}{lll}
u^i(t,X)=-U^i(t,-X),&\quad X>0,&\quad \mbox{for}\quad i=1,...,m\\
u^j(t,X)=-U^j(t,X),&\quad X>0,&\quad \mbox{for}\quad j=m+1,...,m+n
\end{array}\right.
\end{equation}
and we define the convex Hamiltonians
\begin{equation}\label{eq::r7}
  \left\{\begin{array}{lll}
      H_i(p)=-\frac{1}{\gamma^i}f^i(\gamma^i p)&\quad \mbox{for}\quad i=1,...,m\\
      H_j(p)=-\frac{1}{\gamma^j}f^j(-\gamma^j p)&\quad \mbox{for}\quad j=m+1,...,m+n
\end{array}\right.
\end{equation}
Then we deduce from (\ref{eq::r3}) that we have
\begin{equation}\label{eq::r8}
u^k_t + H_k(u^k_X)=0,\quad X>0,\quad \mbox{for}\quad k=1,...,m+n
\end{equation}
with equality of the functions at the origin, i.e.
$$
u^k(t,0)=u(t,0) \quad \mbox{for any}\quad k\in \left\{1,...,m+n\right\}.
$$
Notice that for the choice $V_{\mathrm{max}}=1=\rho_{\mathrm{max}}$
in (\ref{eq::r155}), we get with $f^k(\rho)=f(\rho)=\rho(1-\rho)$ for all
$k\in I_N$, that
$$
\left\{\begin{array}{ll}
 L_{\textup{ref}}(q)=\frac{1}{4}(1+q)^2&\\
 L_i(q)=\frac{1}{\gamma^i}L_{\textup{ref}}(q) &\quad \mbox{for}\quad i=1,...,m\\
 L_j(q)=\frac{1}{\gamma^j}L_{\textup{ref}}(-q) &\quad \mbox{for}\quad j=m+1,...,m+n
\end{array}\right.
$$
In particular this shows that the $L_k(0)$ are not all the same, even in
the simplest case.

\subsection{The junction condition and its interpretation}

A junction condition is still needed so that the solution of
(\ref{eq::r8}) be uniquely defined.  Indeed, at first glance, one may
think that $u_t(t,0)$ is equal to $-g'(t)$ which is given by
(\ref{eq::r4}) (where we have assumed (\ref{eq::r11})).  The point is
that this condition can not be satisfied for every time.  One way to
be convinced oneself of that fact is to consider the case $m=n=1$ with
$f^1=f^2=f$. Then, we look at solutions $u$ of the Hamilton-Jacobi
equation in $\R$ with the artificial junction. We can simply associate
with it the classical conservation law on the whole real line. We can
then consider a single shock moving with constant velocity for the
conservation law.  When this shock will pass through the junction
point (let us say at time $t_0$), this will mean that $u_t(t_0,\cdot)$
is discontinuous in space at the junction point. In particular the
formal computations of Subsection~\ref{s51} are no longer valid at
that time $t_0$, even if they are valid for $t\neq t_0$.  For a
general problem, one may expect that our formal computations are only
valid for almost every time (even if it is not clear for us).

In view of Theorem \ref{th::r4}, if we restrict our attention to continuous
solutions $u$, then we will have $u\le u_{\mathrm{oc}}$ where $u_{\mathrm{oc}}$ is the
solution associated to the optimal control problem.  This shows in
particular that we have
$$
u(t,0)\le u_{\mathrm{oc}}(t,0)
$$
which means (in view of (\ref{eq::r6}) and the interpretation of $-U^l$
given in subsection \ref{s51}) that we have a universal bound on the total
amount of cars passing through the junction point during the time interval
$[0,t)$. If we assume moreover that this amount of cars is maximal, then we
can choose (and indeed have to choose) $u=u_{\mathrm{oc}}$ and the natural junction
condition is then
\begin{equation}\label{eq::r9}
u_t(t,0) + \max_{k=1,...,N} H_k^-(u^k_X(t,0^+)) =0
\end{equation}
with
$$
H_k^-(p)=\sup_{q\le 0} \left(pq - L_k(q)\right)\quad \mbox{and}\quad
L_k(p)=\sup_{q\in\R} \left(pq -H_k(q)\right).
$$
Using our assumption (\ref{eq::r0}) on the functions $f^k$, let us define for $k=1,...,N$
the Demand functions
$$
f^k_D(p)=\left\{\begin{array}{ll}
f^k(p) &\quad \mbox{for}\quad p \le \rho^k_c\\
f^k(\rho^k_c) &\quad \mbox{for}\quad p \ge \rho^k_c
\end{array}\right.
$$
and the Supply functions
$$
f^k_S(p)=\left\{\begin{array}{ll}
f^k(\rho^k_c)&\quad \mbox{for}\quad p \le \rho^k_c\\
f^k(p)  &\quad \mbox{for}\quad p \ge \rho^k_c.
\end{array}\right.
$$
From assumption (\ref{eq::r0}) on the functions $f^k$, 
we deduce that
$$
\left\{\begin{array}{ll}
    H_i^-(p)= -\frac{1}{\gamma^i}f^i_D(\gamma^i p),& \quad \mbox{for}\quad i=1,...,m\\
    H_j^-(p)=  -\frac{1}{\gamma^j}f^j_S(-\gamma^j p),& \quad
    \mbox{for}\quad j=m+1,...,m+n=N.
\end{array}\right.
$$
Condition (\ref{eq::r9}) means that
\begin{multline}\label{eq::r10}
  -U^1_t(t,0)  =u_t(t,0)  = \min_{k=1,...,N} -H_k^-(u^k_X(t,0^+))\\
  =\min\bigg(\min_{i=1,...,m}\frac{1}{\gamma^i}f^i_D(\rho^i(t,0^-)),
    \min_{j=m+1,...,m+n}\frac{1}{\gamma^j}f^j_S(\rho^j(t,0^+))\bigg).
\end{multline}
Notice that from (\ref{eq::r11}), it is natural to compare 
$$
\frac{1}{\gamma^i}f^i(\rho^i(t,0^-)) \quad \text{and} \quad
\frac{1}{\gamma^j}f^j(\rho^j(t,0^+)).
$$
  Then condition (\ref{eq::r10}) is
nothing else that the Demand and Supply condition of Lebacque, which claims
that the passing flux is equal to the minimum between the Demand and the
Supply, as it is defined in \cite{lk} (at least in the case $m=1$).

In the special case $m=1$, it is explained in \cite{lkbis} that this
condition (\ref{eq::r10}) is also equivalent to the condition defining the
Riemann solver at the junction point in the book of Garavello and Piccoli
\cite{gp}. Let us notice that this condition is also related to the Bardos,
Le Roux, N\'edelec \cite{bln} boundary condition.

\section{Viscosity solutions}
\label{sec:visc}

In this section, we consider a weaker assumption than (A1).  We
introduce the following assumption:

\paragraph{(A1')} For each $i\in I_N$, 
\begin{itemize}
\item the function $H_i\colon\R\to \R$ is continuous and 
$\lim_{|p|\to +\infty} H_i(p) = +\infty;$
\item there exists $p^i_0 \in \R$ such that $H_i$ is non-increasing on
  $(-\infty,p^i_0]$ and non-decreasing on $[p^i_0,+\infty)$;
\end{itemize}
When (A1') holds true, the function $H_i^-$ is defined by
$H_i^-(p)=\inf_{q\le 0} H_i(p+q)$. We now make the following useful
remark whose proof is left to the reader.
\begin{lem}\label{lem::r18}
Assumption (A1) implies Assumption (A1').
\end{lem}
Next we give equivalent definitions of viscosity solutions for
\eqref{eq:main}.  We give a first definition where the junction condition
is satisfied in ``the classical sense''; we then prove that it is
equivalent to impose it in ``the generalized sense''. It is essential if
one expects solutions to be stable. 

We give a first definition of viscosity solutions for \eqref{eq:main} in terms
of test functions by imposing the junction condition in the classical
sense.  We recall the definition of the upper and lower semi-continuous
envelopes $u^*$ and $u_*$ of a function $u\colon[0,T)\times J$:
$$
u^*(t,x)=\limsup_{(s,y)\to (t,x)} u(s,y)\quad \mbox{and}\quad
u_*(t,x)=\liminf_{(s,y)\to (t,x)} u(s,y).
$$
\begin{definition}[Viscosity solutions] \label{defi:main}
  A function $u \colon [0,T)\times J \to \R$ is a {\em sub-solution
  (resp. super-solution) of \eqref{eq:main} on $J_T$} if it is upper
  semi-continuous (resp. lower semi-continuous) and if for any 
  $\phi\in C^1_*(J_T)$ such that $u \le \phi$ in $B(P,r)$ for some $P =
  (t,x) \in J_T$, $r>0$ and such that $u=\phi$ at $P \in J_T$, we have
$$
\phi_t (t,x) + H(x, \phi_x(t,x)) \le 0 \quad \text{(resp. $\ge 0$)},
$$
that is to say
\begin{itemize}
\item if $x \in J_i^*$, then
$$
\phi_t (t,x) + H_i (\phi_x (t,x)) \le 0 \quad \text{(resp. $\ge 0$)};
$$
\item if $x =0$, then
\begin{equation}\label{eq:cond-jonction}
\phi_t (t,0) + \max_{i\in I_N} H_i^- (\phi_x^i (t,0)) \le 0 
\quad \text{(resp. $\ge 0$).}
\end{equation}
\end{itemize}

A function $u \colon [0,T)\times J \to \R$ is a {\em sub-solution
  (resp. super-solution) of \eqref{eq:main}-\eqref{eq:ci} on
  $J_T$} if it is a sub-solution (resp. super-solution) of
\eqref{eq:main} on $J_T$ and moreover satisfies $u(0,\cdot)\le u_0$
(resp. $u(0,\cdot)\ge u_0$).

A function $u\colon[0,T)\times J \to \R$ is a {\em (viscosity) solution of
  \eqref{eq:main} (resp. \eqref{eq:main}-\eqref{eq:ci}) on $J_T$} if
$u^*$ is a sub-solution and $u_*$ is a super-solution of \eqref{eq:main}
(resp. \eqref{eq:main}-\eqref{eq:ci}) on $J_T$.
\end{definition}
As mentioned above, the following proposition is important in order to get
discontinuous stability results for the viscosity solutions of
Definition~\ref{defi:main}.
\begin{proposition}[Equivalence with relaxed junction
    conditions] \label{prop:equiv-def} Assume (A1'). A function $u \colon
  J_T \to \R$ is a sub-solution (resp. super-solution) of
  \eqref{eq:main} on $J_T$ if and only if for any function $\phi\in
  C^1_*(J_T)$ such that $u \le \phi$ in $J_T$ and $u=\phi$ at $(t,x)
  \in J_T$,
\begin{itemize}
\item if $x \in J_i^*$, then
$$
\phi_t (t,x) + H_i (\phi_x (t,x)) \le 0 \quad \text{(resp. $\ge 0$)}
$$
\item if $x =0$, then either there exists $i \in I_N$ such that
$$
\phi_t (t,0) + H_i (\phi_x (t,0)) \le 0 \quad \text{(resp. $\ge 0$)}
$$
or
\eqref{eq:cond-jonction} holds true.
\end{itemize}
\end{proposition}
\begin{proof}[Proof of Proposition~\ref{prop:equiv-def}.]
  We classically reduce to the case where the ball $B(P,r)$ is
  replaced with $J_T$.

The ``if'' part is clear. Let us prove the ``only if'' one.  We
distinguish the sub-solution case and the super-solution one.  We start
with super-solutions since it is slightly easier.  \medskip

\noindent \textbf{Case 1: super-solution case.}
We consider a test function $\phi\in C^1_*(J_T)$ such that $u \ge
\phi$ in $J_T$ and $u = \phi$ at $(t_0,x_0)$. There is nothing to prove if
$x_0 \neq 0$ so we assume $x_0=0$. We have to prove that $\phi_t (t_0,0) +
\sup_{i \in I_N} H_i^- (\phi_x^i (t_0,0)) \ge 0$.  We argue by
contradiction and we assume that
\begin{equation}\label{hyp:abs1}
\phi_t (t_0,0) + \sup_{i \in I_N} H_i^- (\phi_x^i (t_0,0)) < 0.
\end{equation}

Then it is easy to see that there exists a function $\tilde{\phi}\in
C^1_*(J_T)$ such that $\phi \ge \tilde{\phi}$ with equality at the
point $(t_0,0)$ and such that
\begin{equation}\label{eq:phitilde1}
\tilde{\phi}^i_x (t_0,0) = \min (\phi^i_x (t_0,0),p_0^i) 
\quad \text{ and } \quad \tilde{\phi}_t (t_0,0) = \phi_t (t_0,0). 
\end{equation}
Notice that
\begin{equation}\label{eq:phitilde2}
H_i^- (\tilde{\phi}^i_x (t_0,0)) 
\le H_i (\tilde{\phi}^i_x (t_0,0)) 
\le H_i^- (\phi_x^i (t_0,0)).
\end{equation}
The first inequality is straightforward.  To check the second
inequality, we have to distinguish two cases.  Either we have
$\tilde{\phi}^i_x (t_0,0) < {\phi}^i_x (t_0,0)$, and then
$\tilde{\phi}^i_x (t_0,0)=p^i_0$ and we use the fact that the minimum of
$H_i^-$ is $H_i(p^i_0)$.  Or $\tilde{\phi}^i_x (t_0,0) = {\phi}^i_x
(t_0,0)$ and then this common value belongs to the interval
$(-\infty,p^i_0]$ on which we have $H_i=H_i^-$.  

Since $u \ge \tilde{\phi}$ in $J_T$ and $u = \tilde{\phi}$ at $(t_0,0)$,
we conclude that either
$$
\tilde{\phi}_t (t_0,0) + \sup_{i \in I_N} H_i^- (\tilde{\phi}_x^i(t_0,0) ) \ge 0
$$
or there exists $i \in I_N$ such that
$$
\tilde{\phi}_t (t_0,0) + H_i (\tilde{\phi}_x^i (t_0,0)) \ge 0.
$$
In view of \eqref{eq:phitilde1} and \eqref{eq:phitilde2}, we obtain
a contradiction with \eqref{hyp:abs1}.\medskip

\noindent \textbf{Case 2: sub-solution case.}
We consider a function $\phi\in C^1(J_T)$ such that $u \le \phi$ in
$J_T$ and $u = \phi$ at $(t_0,x_0)$. There is nothing to prove if $x_0 \neq
0$ and we thus assume $x_0=0$. We have to prove that $\phi_t (t_0,0) +
\sup_{i \in I_N} H_i^- (\phi_x^i (t_0,0)) \le 0$.  We argue by
contradiction and we assume that
\begin{equation}\label{hyp:abs2}
  \phi_t (t_0,0) + \sup_{i \in I_N} H_i^- (\phi_x^i (t_0,0)) > 0.
\end{equation}

In order to construct a test function $\tilde{\phi}$, we first
consider $\bar I_1 \subset I_N$ the set of $j$'s such that
$$
H_j^- (\phi_x^j (t_0,0)) < \sup_{i \in I_N} H_i^- (\phi_x^i (t_0,0)).
$$
Since $H_j$ is coercive, there exists $q^j \ge p^j_0$ such that $H_j
(q^j) = \sup_{i \in I_N} H_i^- (\phi_x^i (t_0,0))$.

We next consider a test function $\tilde{\phi}\in C^1_*(J_T)$ such
that $\phi \le \tilde{\phi}$ with equality at $(t_0,0)$ and such that
\begin{equation}\label{eq:phitilde3}
\tilde{\phi}^i_x (t_0,0) =  \left\{\begin{array}{ll}
\max (\phi^i_x (t_0,0),q^i) & \text{ if } i \in \bar I_1, \\
\phi^i_x (t_0,0) & \text{ if not,}
\end{array}\right.
\quad \text{ and } \quad \tilde{\phi}_t (t_0,0) = \phi_t (t_0,0). 
\end{equation}
Notice that for all $j \in I_N$,
\begin{equation}\label{eq:phitilde4}
  H_j (\tilde{\phi}^j_x (t_0,0)) \ge \sup_{i\in I_N} H_i^- (\tilde{\phi}^i_x (t_0,0))=
  \sup_{i\in I_N} H_i^- ({\phi}^i_x (t_0,0))
\end{equation}
 where for the inequality, we have in particular used the fact that
$H_j$ is non-decreasing on $[p^j_0,+\infty)$.

Since $u \le \tilde{\phi}$ in $J_T$ and $u = \tilde{\phi}$ at $(t_0,0)$,
we conclude that either
$$
\tilde{\phi}_t (t_0,0) + \sup_{i \in I_N} H_i^- (\tilde{\phi}_x^i(t_0,0) ) \le 0
$$
or there exists $j \in I_N$ such that
$$
\tilde{\phi}_t (t_0,0) + H_j (\tilde{\phi}_x^j (t_0,0)) \le 0.
$$
In view of \eqref{eq:phitilde3} and \eqref{eq:phitilde4}, we obtain
a contradiction with \eqref{hyp:abs2}.
This ends the proof of the Proposition.
\end{proof}
We now prove Theorem~\ref{th::r4}. 
\begin{proof}[Proof of Theorem~\ref{th::r4}.]
  Let us consider a function $\phi\in C^1_*(J_T)$ such that $u\le \phi$
  with equality at $(t_0,0)$ with $t_0\in (0,T)$.  Modifying $\phi$ if
  necessary, we can always assume that the supremum of $u-\phi$ is strict
  (and reached at $(t_0,0)$).  For $\eta=(\eta_1,...,\eta_N)\in (\R^+)^N$,
  we set
$$
M_\eta=\sup_{(t,x=x_je_j)\in J_T} \left(u(t,x)-\phi(t,x)
  -\frac{\eta_j}{|x_j|}\right).
$$
Because $u$ is continuous at $(t_0,0)$, we get for $\eta \in
(\R^+_*)^N$ that
\begin{equation}\label{eq::r21}
\left\{\begin{array}{l}
M_\eta \to M_0=0\\
(t^\eta,x^\eta)\to (t_0,0)
\end{array}\right|\quad \mbox{as soon as one of the component $\eta_{i_0} \to 0$.}
\end{equation}
where $(t^\eta,x^\eta)\in J_T$ is a point where the supremum in
$M_\eta$ is reached.

Moreover given the components $\eta_j>0$ for $j\in I_N\setminus
\left\{i_0\right\}$, we can use (\ref{eq::r21}) in order to find
$\eta_{i_0}>0$ small enough to ensure that $x^{\eta}\in J_{i_0}^*$.
Then we have in particular the following sub-solution viscosity
inequality at that point $(t^\eta,x^\eta)$:
$$ 
\phi_t + H_{i_0}\left(\phi_x -\frac{\eta_{i_0}}{|x^\eta|^2}\right)\le
0.
$$
Therefore passing to the limit $\eta_{i_0}\to 0$, we get
$$
\phi_t + H_{i_0}^-(\phi_x^{i_0})\le 0 \quad \mbox{at}\quad (t_0,0).
$$
Because this is true for any  $i_0\in I_N$, we finally get the
sub-solution viscosity inequality at the junction:
$$
\phi_t + \max_{i\in I_N} H_{i}^-(\phi_x^i)\le 0 \quad \mbox{at}\quad
(t_0,0).
$$
Now the fact that $u \le u_{\mathrm{oc}}$ follows from the comparison
principle. This ends the proof of the Theorem.
\end{proof}

\section{The minimal action}
\label{sec:oc}

We already mentioned that the optimal control solution of the
Hamilton-Jacobi equation defined by \eqref{eq::r17} plays a central role in
our analysis. We remark that for $x \in J$ and $t>0$,
\begin{equation}\label{eq::r76}
  u_{\mathrm{oc}}(t,x)=\inf_{y\in J}\left\{u_0(y)
    + \dz (0,y;t,x) \right\}
\end{equation}
where 
$$
\dz (0,y;t,x) = \min_{X\in {\mathcal A}(0,y;t,x)}
\int_0^t L(X(\tau),\dot{X}(\tau))d\tau.
$$
More generally, keeping in mind the weak KAM theory, we define the
so-called \emph{minimal action} $\dz :\{(s,y,t,x) \in ([0,\infty)
\times J)^2, s < t \} \to \R$ by
\begin{equation}\label{eq:def-D}
\dz(s,y;t,x)=\inf_{X\in {\mathcal A}(s,y;t,x)} \int_s^t
L(X(\tau),\dot{X}(\tau))d\tau. 
\end{equation}
It is convenient to extend $\dz$ to $\{s=t\}$. We do so by setting
$$
\dz (t,y,t,x) = \begin{cases} 0 & \text{ if } y = x, \\ +\infty &
  \text{ if } y \neq x. \end{cases}
$$
\begin{rem}[Dynamic Programming Principle]
  Under assumptions (A0)-(A1), it is possible (and easy) to prove
  the following Dynamic Programming Principle: for all $x\in J$ and
  $s\in [0,t]$,
$$
u_{\mathrm{oc}}(t,x)=\inf_{y\in J} \left\{u_{\mathrm{oc}}(s,y) + \dz(s,y;t,x)\right\}.
$$
Notice that a super-optimality principle will be proved in
Proposition~\ref{prop:superopt}.
\end{rem}
Before stating the main result of this section, we 

The following result can be considered as the core of our analysis.  The
most important part of the following theorem is the fact that the minimal
action is semi-concave with respect to $(t,x)$ (resp. $(s,y)$).
\begin{theorem}[Key inequalities for $\dz$]
  \label{thm:minimal-action} $\dz$ is finite, continuous in $\{
  (s,y;t,x)\colon 0<s<t, x,y \in J\}$ and lower semi-continuous in $\{
  (s,y;t,x)\colon 0< s \le t, x,y \in J\}$.  Moreover, for all
  $(s_0,y_0)$ and $(t_0,x_0)\in (0,T)\times J$, $s_0<t_0$, there exist
  two functions $\phi, \psi \in C^1_*(J_T)$ and $r>0$ such that
\begin{itemize}
\item 
$\phi \ge \dz(s_0,y_0;\cdot,\cdot)$ on a ball $B(P_0,r)$  with equality at $P_0=(t_0,x_0)$ and
\begin{equation}\label{eq::r60}
\phi_t + H(x_0,\phi_x)\ge 0 \quad \mbox{at}\quad (t_0,x_0).
\end{equation}
\item 
$\psi \ge \dz(\cdot,\cdot;t_0,x_0)$ on a ball $B(Q_0,r)$  with equality at  $Q_0=(s_0,y_0)$ and
\begin{equation}\label{eq::r61}
\left\{\begin{array}{ll}
-\psi_s + H(y_0,-\psi_y)\le 0 \quad \mbox{at}\quad (s_0,y_0) &\quad \mbox{if}\quad N\ge 2,\\
-\psi_s + H_1(-\psi_y)\le 0\quad \mbox{at}\quad (s_0,y_0) &\quad \mbox{if}\quad N=1.
\end{array}\right.
\end{equation}
\end{itemize}
Moreover, for all $R>0$, there exists a constant $C_R>0$ such that we have
\begin{equation}\label{eq::r73}
  d(y_0,x_0)\le R \quad \Longrightarrow\quad |\phi_x(t_0,x_0)|+ |\psi_y(s_0,y_0)|\le C_R.
\end{equation}
\end{theorem}
\begin{rem}
  As we shall see when proving this result, we can even require
  equalities instead of inequalities in \eqref{eq::r60} and \eqref{eq::r61}.
\end{rem}
Since the proof of Theorem~\ref{thm:minimal-action} is lengthy and
technical, we postpone it until Section~\ref{sec:complete}. When
proving the main results of our paper in the next section, we also
need the following lower bound on $\dz$. We remark that this bound
ensures in particular that it is finite. 
\begin{lemma}\label{lem:am}
Assume (A1). Then
$$
\dz (s,y;t,x) \ge  \frac{\gamma}{4(t-s)} d(y,x)^2 - C_0 (t-s) 
$$
where $C_0\colon=\max(0, -L_0 (0)+\frac{\gamma_0^2}{\gamma})$, $\gamma$
appears in (A1), $ \gamma_0 = \max_{i\in I_N} |L'_i (0)|$ and $L_0(0)$
is chosen as in (\ref{eq::r26}). Moreover,
$$
\dz (s,x;t,x) \le L_0 (0) (t-s).
$$
In particular, if $(t_n,x_n) \to (t,x)$, then $\dz (t_n,x_n;t,x_n) \to
0$ as $n \to \infty$.
\end{lemma}
\begin{proof}[Proof of Lemma~\ref{lem:am}.]
We only prove the first inequality since the remaining of the
statement is elementary.  We have
$$
  L_i (p)    \ge \frac{\gamma}2 p^2 + L_i' (0) p + L_i (0) \ge
  \frac{\gamma}2 p^2 -\gamma_0 |p| + L_0 (0)
 \ge \frac{\gamma}4 p^2+ L_0 (0) -\frac{\gamma_0^2}{\gamma}.
$$
This shows that
\begin{equation}\label{eq::r50}
L_i (p) \ge  \frac{\gamma}4 p^2 -C_0.
\end{equation}
Thus we can write for $X (\cdot) \in \mathcal{A}(s,y;t,x)$,
$$
  \int_s^t L(X(\tau),\dot{X}(\tau))\, d\tau \ge  - C_0 (t-s)
  + \frac{\gamma}4  \int_s^t (\dot{X}(\tau))^2\, d\tau.
$$
Then Jensen's inequality allows us to conclude.
\end{proof}

\section{Proofs of the main results}
\label{sec:op}

In this section, we investigate the uniqueness of the solution of
\eqref{eq:main}-\eqref{eq:ci}.  In particular, we will show that the
solution constructed by Perron's method coincide with the function
$u_{\mathrm{oc}}$ coming from the associated optimal control problem.

\subsection{Super-solutions and super-optimality}

In this subsection, we will show that a super-solution satisfies a
super-optimality principle.  For the sake of clarity, we first give a
formal argument to understand this claim. We consider the auxiliary
function, for $s \le t$,
\begin{equation}\label{def:U}
  U_{t,x} (s) = \inf_{y \in J} \{ u(s,y) + \mathcal{D} (s,y;t,x) \}
\end{equation}
and we are going to explain formally that it is non-decreasing with
respect to $s$ as soon as $u$ is a super-solution of
\eqref{eq:main}. We call this property a super-optimality principle.
Notice that this is strongly related to the fact that the quantity
$U_{t,x}(s)$ is constant in $s$ if $u$ is equal to the optimal control
solution $u_{\mathrm{oc}}$.

Assume that the infimum defining $U$ is attained for some $\bar{y} \in
J$.  Then we write
\begin{eqnarray*}
U_{t,x}' (s) &=& \partial_s u (s,\bar{y}) + \partial_s \mathcal{D}
(s,\bar y;t,x) \\
  \partial_x u (s,\bar{y}) &=& - \partial_y \mathcal{D} (s, \bar y;t,x).
\end{eqnarray*}
Moreover assuming $\dz$ to be smooth (which is not the case), we
formally get from (\ref{eq::r61}) the fact that $\partial_s
\mathcal{D}(s,\bar{y};t,x) \ge H(\bar{y},-\partial_y \mathcal{D} (\bar
s,\bar{y};t,x) )$ (at least in the case $N\ge 2$). Hence
$$
U_{t,x}' (s) \ge \partial_s u (s,\bar{y}) + H(\bar{y},\partial_x u
(s,\bar{y})) \ge 0.
$$
We thus conclude that $U_{t,x}$ is non-decreasing if $u$ is a
super-solution of \eqref{eq:main}. We now give a precise statement and
a rigorous proof.
\begin{proposition}[Super-optimality of super-solutions]\label{prop:superopt}
  Assume (A1).  Let $u\colon [0,T)\times J \to \R$ be a super-solution of
  \eqref{eq:main} on $J_T$ such that there exists $\sigma >0$ such
  that for all $(t,x) \in J_T$,
\begin{equation}\label{eq:superlin-}
u(t,x) \ge -\sigma (1+d(x,0))
\end{equation}
Then for all
$(t,x) \in J_T$ and $s \in (0,t]$,
\begin{equation}\label{eq:superopt}
u(t,x) \ge \inf_{y \in J} \{ u (s,y) + \mathcal{D}(s,y;t,x) \} 
\end{equation}
Assume moreover (A0) and that $u$ is a super-solution of
(\ref{eq:main})-(\ref{eq:ci}) on $J_T$.  Then we have $u\ge u_{\mathrm{oc}}$ on
$[0,T)\times J$.
\end{proposition}
\begin{proof}[Proof of Proposition~\ref{prop:superopt}.]
The proof proceeds in several steps. \medskip

\noindent \textbf{Step 1: preliminary.}
Notice first that from (\ref{eq::r41}), we get
$$
u(s,y)+\dz(s,y;t,x)\ge \frac{\gamma}{4(t-s)} d(y,x)^2 - C_0 (t-s)
-\sigma(1+|y|).
$$
Using the lower semi-continuity of $\dz$, we see that the
infimum in $y$ of this function is then reached for bounded $y$'s. Moreover
by lower semi-continuity of the map $(s,y;t,x)\mapsto u(s,y)+\dz(s,y;t,x)$,
we deduce in particular that the map $(s;t,x)\mapsto U_{t,x}(s)$ (and then
also $s\mapsto U_{t,x}(s)$) is lower semi-continuous. \medskip

\noindent \textbf{Step 2: the map $\mathbf{s\mapsto U_{t,x}(s)}$ is
  non-decreasing.}
We are going to prove that for $s \in (0,t)$, $ U_{t,x}'(s) \ge 0$ in
the viscosity sense. We consider a test function $\varphi$ touching
$U_{t,x}$ from below at $\bar s\in (0,t)$.  There exists $\bar y$ such
that
$$
U_{t,x}(\bar s) = u(\bar s,\bar y) + \mathcal{D}(\bar s,\bar y;t,x).
$$
We deduce from the definition of $U_{t,x}$ that 
$$
\varphi(s)- \mathcal{D} (s,y;t,x) - [\varphi (\bar s) -
\mathcal{D}(\bar s,\bar y;t,x)] \le u(s,y) - u (\bar s, \bar y).
$$
By Theorem~\ref{thm:minimal-action}, there exists a test function
$\psi$ such that $\psi\ge \dz(\cdot,\cdot ;t,x)$ on a ball $B(\bar
Q,r)$ with equality at $\bar Q = (\bar s,\bar y)$.  Hence, we can
rewrite the previous inequality by replacing $\mathcal{D}$ with
$\psi$. We then obtain that $(s,y) \mapsto \varphi (s) - \psi(s,y)$ is
a test function touching $u$ at $(\bar s, \bar y)$ from below. Since
$u$ is a super-solution of \eqref{eq:main}, we have in the cases $N\ge
2$ or $N=1$ and $\bar y\not= 0$
$$
\varphi'(\bar s) \ge \psi_s (\bar s,\bar y) - H(\bar y,-\partial_y
\psi(\bar s,\bar y))\ge 0
$$
and in the case $N=1$ and $\bar y=0$
$$
\varphi'(\bar s) \ge \psi_s (\bar s,\bar y) - H_1^-(-\partial_y
\psi(\bar s,\bar y))\ge \psi_s (\bar s,\bar y) - H_1(-\partial_y
\psi(\bar s,\bar y))\ge 0
$$
where we used the properties of the function $\psi$ given by
Theorem~\ref{thm:minimal-action}. \medskip

\noindent \textbf{Step 3: conclusion.}
Let us define for $(t,x)\in J_T$ the following kind of lower
semi-continuous envelope (for the past in time)
$$
\underline{u}(t,x)=\liminf\{ u(t_n,x_n):(t_n,x_n)\to (t,x), t_n< t \}.
$$
Let us notice that we have
\begin{equation}\label{eq::r71}
\underline{u}_* = u_*=u \quad \mbox{on}\quad J_T.
\end{equation}
Given a point $(t,x)\in J_T$, let us consider a sequence $(t_n,x_n)
\to (t,x)$ such that
$$
\underline{u}(t,x)=\lim_{n\to +\infty} u(t_n,x_n).
$$
Using Lemma~\ref{lem:am}, we have for any $s<t_n<t$
$$
U_{t,x_n}(s) \le U_{t,x_n}(t_n)\le u(t_n,x_n) + \dz(t_n,x_n,t,x_n) \to
\underline{u}(t,x).
$$
  Therefore from the lower
semi-continuity of $U$, we get
$$
U_{t,x}(s)\le \underline{u}(t,x).
$$
Again from the lower semi-continuity of the map $(t,x)\mapsto
U_{t,x}(s)$, we get passing to the lower semi-continuous envelopes in
$(t,x)$:
$$
U_{t,x}(s)\le \underline{u}_*(t,x)=u(t,x)
$$
where we have used (\ref{eq::r71}). This shows (\ref{eq:superopt}) for
$0<s<t$.  This is still true for $s=t$ by definition of $\dz$.  The
proof is now complete.
\end{proof}

\subsection{Comparison with sub solutions}

\begin{proposition}[Comparison with
    sub solutions] \label{prop:subopt} Let $u\colon J_T \to \R$ be a
  sub-solution of \eqref{eq:main}-(\ref{eq:ci}) on $J_T$, such that
  there exists $\sigma >0$ such that for all $(t,x) \in J_T$,
\begin{equation}\label{eq:subopt-cond}
u(t,x) \le \sigma (1 + d(x,0)).
\end{equation}
Then we have $u \le u_{\mathrm{oc}}$ on $J_T$.
\end{proposition}

In order to prove Proposition~\ref{prop:subopt}, we first state and prove
two lemmas.
\begin{lemma}\label{lem::r46} Assume
  (A0)-(A1). Then the function $u_{\mathrm{oc}}$ defined in (\ref{eq::r17})
  satisfies
$$
|u_{\mathrm{oc}}(t,x)-u_0(x)|\le Ct.
$$
\end{lemma}
\begin{proof}[Proof of Lemma~\ref{lem::r46}.]
  We first get a bound from below.  Using (\ref{eq::r41}), we deduce
  (denoting by $L_{u_0}$ the Lipschitz constant for $u_0$):
$$
\begin{array}{ll}
u_0(y)+ \dz(s,y;t,x) & \ge u_0(x)+ \frac{\gamma}{4t} (d(y,x))^2  -C_0 t -L_{u_0}d(y,x)\\
\\
&\ge u_0(x) - C_2t 
\end{array}
$$
with
$$
-C_2 = \inf_{a\in [0,+\infty)} \left\{\frac{\gamma}{4}a^2  -C_0
  -L_{u_0}a \right\}> -\infty.
$$
This implies that
$$
u_{\mathrm{oc}}(x)\ge u_0(x) - C_2t.
$$

We next get a bound from above.  We have
$$
u_{\mathrm{oc}}(x)\le u_0(x)+ \dz(0,x;t,x)\le u_0(x) + Mt
$$
with
$$
M=\sup_{i\in I_N} L_i(0).
$$
This ends the proof of the lemma.
\end{proof}
\begin{lemma}\label{lem::r72}
  Assume (A0)-(A1). Let $u: [0,T)\times J\to \R$ be a sub-solution of
  \eqref{eq:main}-(\ref{eq:ci}) on $J_T$, satisfying
  \eqref{eq:subopt-cond}. Then there exists a constant $C>0$ such that
\begin{equation}\label{eq:superlin+}
  u(t,x) \le u_0 (x) + C t \quad \mbox{for all}\quad  (t,x)\in J_T.
\end{equation}
\end{lemma}
\begin{proof}[Proof of Lemma~\ref{lem::r72}.]
Using the Lipschitz regularity of $u_0$, we can easily consider a
smooth approximation $u_0^\eta$ of $u_0$ such that $u_0^\eta \ge u_0$
and $|u_0^\eta-u_0|_{L^\infty(J)} \to 0$ as $\eta\to 0$.  Then
consider the following supremum for $\eta,\alpha >0$
$$
N_{\eta,\alpha}= \sup_{(t,x) \in [0,T)\times J} \{ u(t,x) - u_0^\eta
(x) -Ct- \alpha d(x,0)^2 - \frac{\eta}{T-t} \}.
$$
We claim that $N_{\eta,\alpha} \le 0$ for some $C$ large enough
independent on $\eta,\alpha>0$ small enough.  The lemma will be
obtained by letting $\alpha$ and $\eta$ go to $0$. We argue by
contradiction and assume that $N_{\eta,\alpha} >0$. Thanks to
\eqref{eq:subopt-cond}, the supremum $N_{\eta,\alpha}$ is attained for
some $(t,x) \in [0,T)\times J$. If $t=0$, we have $N_{\eta,\alpha} \le
0$. Therefore $t>0$ and we can use the fact that $u$ is a sub-solution
to obtain for $x=x_ie_i$
$$
\frac{\eta}{T^2} +C -\max_{j\in I_N} L_j (0)  \le 
\frac{\eta}{T^2} +C + H(x,\partial_x u_0^\eta (x) + 2 \alpha x_i) \le 0
$$
where we have used assumption (A1) to estimate $H$ from below.  Notice
that we have also made use of a slight abuse of notation in the case
$x =0$.  Choosing $C=\max_{j\in I_N} |L_j (0)|$ allows us to conclude
to a contradiction.  This ends the proof of Lemma~\ref{lem::r72}.
\end{proof}
We now turn to the proof of Proposition~\ref{prop:subopt}.
\begin{proof}[Proof of Proposition~\ref{prop:subopt}.]
The proof proceeds in several steps. \medskip

\noindent \textbf{Step 1: preliminaries.}
Let us consider
$$
M=\sup_{(t,x)\in [0,T)\times J} \left\{u(t,x)-u_{\mathrm{oc}}(t,x)\right\}.
$$
From Lemmas~\ref{lem::r46} and \ref{lem::r72}, we deduce that we have
$M\le 2CT <+\infty$. We want to prove that $M\le 0$.

To this end, we perform the usual corrections considering
the following supremum for $\eta,\alpha>0$
$$
M_{\eta,\alpha}= \sup_{(t,x)\in [0,T)\times J} \left\{ u (t,x) -
  u_{\mathrm{oc}}(t,x)- \alpha d(x,0)^2 - \frac{\eta}{T-t} \right\}.
$$

As it is proved classically, we also have that $M_{\eta,\alpha} \to
M_{\eta,0}$ as $\alpha \to 0$ where
$$
M_{\eta,0}= \sup_{(t,x)\in [0,T)\times J} \left\{ u (t,x) - u_{\mathrm{oc}}(t,x) 
 - \frac{\eta}{T-t}\right\}. 
$$
We argue by contradiction by assuming that $M>0$ and then
$M_{\eta,0}\ge M/2 >0$ for $\eta>0$ small enough and fixed for the
rest of the proof. \medskip

\noindent \textbf{Step 2: reduction to $\mathbf{\bar t>0}$.}  Notice
that the supremum $M_{\eta,\alpha}$ is achieved for points $(\bar
t,\bar x) \in [0,T) \times J$. Using again Lemmas~\ref{lem::r46} and
\ref{lem::r72}, we also deduce that
$$
M/2 <M_{\eta,0} \le M_{\eta,\alpha} + o_\alpha(1) \le 2C \bar t
$$
and hence $\bar t \ge \frac{M}{4 C}>0$ for $\alpha$ small enough.
\medskip

\noindent \textbf{Step 3: a priori bounds.}
Using the argument of Step~1 of the proof of Proposition
\ref{prop:superopt}, we see that there exists $\bar y\in J$ such that
$$
u_{\mathrm{oc}}(\bar t,\bar x)= u_0(\bar y) + \dz(0,\bar y; \bar t,\bar x).
$$
Therefore we can rewrite $M_{\eta,\alpha}$ as
$$
M_{\eta,\alpha}= \sup_{0\le t<T,x,y \in J} \{ u (t,x) - u_0 (y) - \mathcal{D} (0,y;t,x) 
- \alpha d(x,0)^2 - \frac{\eta}{T-t} \}. 
$$
and the supremum is achieved for $(\bar t,\bar x,\bar y)\in
(0,T)\times J^2$.  Notice that this supremum looks like the classical
one for proving the comparison principle for viscosity solutions, with
the usual penalization term $(y-x)^2/\varepsilon$ replaced here by the
function $\mathcal{D} (0,y;t,x)$.

In view of the bound (\ref{eq::r41}) from below on $\dz$ and
\eqref{eq:superlin+}, we derive from $M_{\eta,\alpha}>0$ that
$$
\frac{\eta}{T-\bar t}+ \alpha d (\bar x,0)^2 +
\frac{\gamma}{4\bar t} d(\bar y,\bar x)^2 \le C_0 \bar t + C\bar t + L_{u_0} d(\bar y,\bar
x)
$$
where $L_{u_0}$ denotes the Lipschitz constant of $u_0$. We conclude
that there exists ${C}_T$ such that
\begin{equation}\label{eq:penal}
\alpha  d(\bar x,0)^2 \le {C}_T \quad \text{ and } \quad d(\bar y,\bar x) \le {C}_T
\end{equation}
where ${C}_T$ depends on $T$,  $C_0$,$ C$, $L_{u_0}$ and $\gamma$.
\medskip

\noindent \textbf{Step 4: getting the viscosity inequality.}
Since $\bar t >0$, we have in particular that
\begin{multline*}
  u(t,x) - \left( \mathcal{D}(0,\bar y;t,x) + \alpha d(x,0)^2
    +\frac{\eta}{T-t} \right) \\
  \le u(\bar t,\bar x) - \left( \mathcal{D}(0,\bar y;\bar t,\bar x) +
    \alpha d(\bar x,0)^2 +\frac{\eta}{T-\bar t}\right).
\end{multline*}
By Theorem~\ref{thm:minimal-action}, there exists a test function
$\phi$ such that $\phi\ge \dz(0,\bar y; \cdot,\cdot)$ on a ball
$B(\bar P,r)$ with equality at $\bar P=(\bar t,\bar x)$.  Hence, we
can rewrite the previous inequality by replacing $\mathcal{D}$ with
$\phi$. We then obtain that $(t,x) \mapsto \phi(t,x) + \alpha d(x,0)^2
+ \frac{\eta}{T-t}$ touches $u$ from above at $(\bar t, \bar x)$ with
$\bar t>0$.  We use next that $u$ is a sub-solution of \eqref{eq:main}
and get for $\bar x =\bar x_i e_i$
$$
\frac{\eta}{T^2} + \phi_t (\bar t,\bar x) + H (\bar x,\phi_{x} (\bar
t,\bar x) + 2\alpha \bar x_i) \le 0
$$
where we have made use of a slight abuse of notation in the case $\bar
x =0$.  On the other hand, we have
$$
\phi_t (\bar t,\bar x) 
+ H (\bar x,\phi_{x} (\bar t,\bar x)) \ge 0
$$
therefore
$$
\frac{\eta}{T^2} + H (\bar x,\phi_{x} (\bar t,\bar x) + 2\alpha \bar
x_i) - H (\bar x,\phi_{x} (\bar t,\bar x)) \le 0.
$$
On the one hand, from (\ref{eq:penal}), we have $0\le \alpha x_i \le
\sqrt{\alpha C_T}$.  On the other hand, we can use (\ref{eq:penal})
and (\ref{eq::r73}) in order to conclude that
$$
|\phi_x (\bar t,\bar x)| \le \bar C
$$
for some constant $\bar C$ which does not depend on $\alpha$.  We can
now use the fact that the Hamiltonians are locally Lipschitz
continuous in order to get the desired contradiction for $\alpha$
small enough.  This ends the proof of the Proposition.
\end{proof}

\subsection{Proof of the main results}

In this subsection, we prove the main results announced in the
introduction.

\begin{proof}[Proof of Theorem~\ref{th::r1}.]
  We simply apply Propositions ~\ref{prop:superopt} and \ref{prop:subopt}
  and get $u\le u_{\mathrm{oc}}\le v$ which implies the result.
\end{proof}
In order to prove Theorem~\ref{th::r2}, we should first prove that solutions are
Lipschitz continuous.
\begin{lemma} \label{pro::r74} Assume (A0)-(A1).  Let $u$ be a solution of
  \eqref{eq:main}-\eqref{eq:ci} on $J_T$.  Then $u$ is Lipschitz continuous
  with respect to $(t,x)$ on $J_T$.
\end{lemma}
\begin{proof}[Proof of Lemma~\ref{pro::r74}.]
  We first recall (see Lemma~\ref{lem::r18}) that (A1) implies (A1').
  We know that the solution $u=u^*$ given by Theorem~\ref{thm:perron}
  satisfies for some constant $C>0$ and all $(t,x) \in J_T$,
$$
|u(t,x) - u_0 (x) | \le Ct.
$$
From the comparison principle (Theorem~\ref{th::r1}), we deduce that
$u=u^*\le u_*$ and then the solution $u=u^*=u_*$ is continuous.

For $h_0>0$ small (with $h_0<T$), we now consider $h\in (0,h_0)$ and
$$
v(t,x) = u(t+h,x) - \sup_{x \in J} (u(h,x) - u_0(x)).
$$
This new function satisfies in particular $v(0,x) \le u_0 (x)$.
Therefore $v$ is a sub-solution of \eqref{eq:main}-\eqref{eq:ci} on
$J_{T-h_0}$.  We thus conclude from the comparison principle that
$v(t,x) \le u(t,x)$, which implies
$$
u(t+h,x) \le u(t,x) + Ch \quad \mbox{for all}\quad (t,x)\in J_{T-h_0}.
$$
Arguing similarly, we can prove that $u(t+h,x) \ge u(t,x) - Ch$.
Because $h_0$ can be chosen arbitrarily small, we conclude that $u$ is
Lipschitz continuous with respect to time on the whole $J_T$.

Since $u$ is a viscosity solution of \eqref{eq:main}, it satisfies in
particular (in the viscosity sense) for each $i\in I_N$:
$$
H_i (u_x) \le C \quad \mbox{on}\quad  (0,T)\times J_i^*.
$$
This implies that there exists a constant $\tilde{C}$ such that (in
the viscosity sense)
$$
|u_x | \le \tilde{C} \quad \mbox{on}\quad  (0,T)\times J^*.
$$
This implies that $u$ is Lipschitz continuous with respect to the space
variable. This ends the proof of the lemma.
\end{proof}
We now turn to the proof of Theorem~\ref{th::r2}. 
\begin{proof}[Proof of Theorem~\ref{th::r2}.]
  The uniqueness of the solution follows from Theorem~\ref{th::r1}.  The
  existence is obtained thanks to the optimal control interpretation
  ($u_{\mathrm{oc}}$ is a solution). The Lipschitz regularity was proved in
  Lemma~\ref{pro::r74} above.  The proof of Theorem~\ref{th::r2} is now complete.
\end{proof}
\begin{proof}[Proof of Theorem~\ref{th::r3}.]
  The fact that the solution is equal to $u_{\mathrm{oc}}$ follows
  from Propositions~\ref{prop:superopt} and \ref{prop:subopt}.  The
  representation formula~(\ref{eq::r77}) follows from (\ref{eq::r76}).
\end{proof}

\section{A complete study of the minimal action}
\label{sec:complete}

\subsection{Reduction of the study}

We start this section with the following remark: the analysis can be
reduced to the case $(s,t)=(0,1)$. Precisely, using the fact that the
Hamiltonian does not depend on time and is positively homogeneous with
respect to the state, the reader can check that a change of variables
in time yields the following
\begin{lemma}\label{lem:reduced}
 For all $x,y \in J$ and $s<t$, we have
\begin{equation}\label{eq::r40}
\dz(s,y;t,x)=(t-s)\dz\left(0,\frac{y}{t-s};1,\frac{x}{t-s}\right).
\end{equation}
\end{lemma}
This is the reason why we consider the reduced minimal action $\dzo \colon
J^2 \to \R$ defined by
$$
\dzo (y,x) = \dz(0,y;1,x).
$$
Thanks to the previous observation, it is enough to prove the
following theorem in order to get Theorem~\ref{thm:minimal-action}.
\begin{theorem}[Key equalities for $\mathbf{\dzo}$]\label{thm:minimal-action-0}
  Let us assume (A1). Then for all $y,x\in J$, the $\dzo(y,x)$ is
  finite. It is continuous in $J^2$ and for all $y,x \in J$, there exists a function
  $\varphi_0\in C^1_*(J^2)$ such that $\varphi_0 \ge \dzo$ on $J^2$,
  $\varphi_0 (y,x) = \dzo (y,x)$ and we have
\begin{align}
\label{eq:dzo-fwd-int}
\text{ if } x \neq 0\colon &(\varphi_0 - x \partial_x \varphi_0  - y \partial_y \varphi_0)
(y,x) + H(x,\partial_x \varphi_0 (y,x))=0 \\
\label{eq:dzo-fwd-bd}
\text{ if } x = 0\colon &
(\varphi_0 - x \partial_x \varphi_0  - y \partial_y \varphi_0)(y,0) 
+ \sup_{i\in I_N} H_i^-(\partial_x^i \varphi_0 (y,0))=0  
\end{align}
and if $y \neq 0$,
\begin{equation}
\label{eq:dzo-bwd-int}
(\varphi_0 - x \partial_x \varphi_0  - y \partial_y \varphi_0)
(y,x) + H(y,-\partial_y \varphi_0(y,x)) = 0
\end{equation} 
and if $y=0$, 
\begin{equation}
\label{eq:dzo-bwd-bd}
\left\{\begin{array}{rl}
 (\varphi_0 - x \partial_x \varphi_0- y \partial_y \varphi_0)(0,x) 
+ \sup_{j\in I_N}  H^-_j (-\partial_y^j \varphi_0(0,x)) = 0 &  \mbox{ if } N\ge 2,\\
 (\varphi_0 - x \partial_x \varphi_0- y \partial_y \varphi_0)(0,x) 
+ H_1(-\partial_y^j \varphi_0(0,x)) = 0 &  \mbox{ if } N=1.
\end{array}\right.
\end{equation}
Moreover, for all $R >0$, there exists $C_R>0$ such that for all $x,y \in J$, 
\begin{equation}\label{eq:grad-control}
d(y,x) \le R \Rightarrow |\partial_x \varphi_0 (y,x)|+|\partial_y \varphi_0 (y,x)|
\le C_R. 
\end{equation}
\end{theorem}
\begin{rem}
  If $I_0=I_N$, then we have $\dzo\in C^1_*(J^2)$.  This good case
  corresponds to the case where all the $L_i(0)$'s are equal.
\end{rem}
We can interpret Lemma~\ref{lem:am} as follows.
\begin{lemma}
Assume (A1). Then
\begin{equation}\label{eq::r41}
   \dzo(y;x) \ge \frac{\gamma}{4} d(y,x)^2 - C_0
\end{equation}
where constants are made precise in Lemma~\ref{lem:am}.
\end{lemma}

\subsection{Piecewise linear trajectories}

We are going to see that the infimum defining the minimal action can
be computed among piecewise linear trajectories, and more precisely
among trajectories that are linear as long as they do not reach the
junction point. This is a consequence of the fact that the
Hamiltonians do not depend on $x$ and are convex (through Jensen's
inequality). 

In order to state a precise statement, we first introduce that optimal
curves are of two types: either they reach the junction point, or they
stay in a branch and are straight lines.  This is the reason why we
introduce first the action associated with straight line trajectories
$$
\dzq(y,x)= \left\{\begin{array}{ll} L_i\left({x_i-y_i}\right) &\quad
    \mbox{if}
    \quad (y,x)\in J_i^2 \setminus \{(0,0)\},\\
    L_0(0)& \quad \mbox{if}\quad y=0=x,\\
    +\infty &\quad \mbox{otherwise}
\end{array} \right.
$$
and the action associated with piecewise linear trajectories passing
through the junction point
$$
\dzun(y,x)=\inf_{0\le \tau_1\le \tau_2\le 1}
\left\{\mathcal{E}_1(\tau_1,y)+\mathcal{E}_2(\tau_2,x)\right\}
$$
where 
$$
\mathcal{E}_1 (\tau_1,y) = 
\begin{cases}
          \tau_1 L_j \left(-\frac{y_j}{\tau_1}\right) -\tau_1 L_0 (0) &
           \text{ for }  y=y_j e_j\not= 0, \tau_1 \neq 0\\
          0  & \text{ for }  y=0 \\
          +\infty & \text{ for } \tau_1 =0, y \neq 0
\end{cases}
$$
and 
$$
\mathcal{E}_2 (\tau_2,x) =  \begin{cases}
    (1-\tau_2) L_i \left(\frac{x_i}{1-\tau_2}\right) +\tau_2 L_0 (0), &
    \text{ for } x=x_i e_i\not= 0, \tau_2 \neq 1\\
    L_0(0)  &\text{ for } x=0 \\
+ \infty & \text{ for } \tau_2 =1, x \neq 0.
\end{cases}
$$
\begin{remark}
  By defining the $\mathcal{E}_i$'s in such a way, we treat the
  degenerate cases: $x=0$ or $y=0$. Indeed, $\tau_1$ (resp. $\tau_2$)
  measures how long it takes to the trajectory to reach the junction
  point (resp. the final point $x$) from the starting point $y$
  (resp. the junction point).
\end{remark}
The following facts will be used several times.
\begin{lemma}\label{lem:cont-ei}
The function $\mathcal{E}_1$ (resp. $\mathcal{E}_2$) is continuous in
$(0,1] \times J^*$ (resp. in $[0,1) \times J^*$). 
\end{lemma}
\begin{lemma}\label{lem:lsc-ei}
The function $\mathcal{E}_i$, $i=1,2$ are lower semi-continuous in $[0,1]
\times J$. 
\end{lemma}
\begin{proof}
Consider the function defined for $(\tau,y)\in [0,1]\times J$ by
$$
g(\tau,y)=\left\{\begin{array}{ll}
    \tau L_i(-\frac{y_i}\tau) & \text{ if } y=y_ie_i\not=0, \tau \neq 0\\
    \tau L_0(0) & \text{ if } y=0 \\
+ \infty & \text{ if } y \neq 0, \tau =0.
\end{array}\right.
$$
From the inequality for $\tau>0$ (consequence of (\ref{eq::r50})):
$$
g(\tau,y)\ge \frac{\gamma}4 \frac{|y|^2}{\tau} - C_0 \tau,
$$
we deduce that $g$ is lower semi-continuous. Consequently, the map
$\mathcal{E}_1$ is lower semi-continuous.  We proceed similarly for
$\mathcal{E}_2$.
\end{proof}
We first show the main lemma of this subsection.
\begin{lemma}\label{lem:optim-curv}
The infimum defining the reduced minimal action $\dzo$ can be computed among
piecewise linear trajectories; more precisely, for all $x,y \in J$, 
$$
\dzo(y,x)=\min\left(\dzq(y,x),\dzun(y,x)\right).
$$
\end{lemma}
\begin{proof}
  We write with obvious notation $\dzo (y,x) = \inf_{X \in
    \mathcal{A}_0(y,x)} \mathcal{E} (X)$. In order to prove the lemma,
  it is enough to consider a curve $X \in \mathcal{A}_0 (y,x)$ and
  prove that
$$
\mathcal{E} (X) \ge \min (\dzq(y,x),\dzun (y,x)).
$$
  To do so, we first remark that the uniform convexity of $L_i$
  implies that for all $p_0 \in \R$, we have
\begin{equation}\label{eq::r42}
L_i(p)\ge L_i(p_0) + L_i'(p_0) (p-p_0) + \frac{\gamma}2 (p-p_0)^2.
\end{equation}
We now consider an admissible trajectory $X\colon[0,1] \to J$ and we treat
different cases. \medskip

\noindent \textsc{Case A: $X((t_1,t_2)) \subset J_i^*$.}  We assume
that a curve $X$ stays in one of the branch $J_i^*$ on the time
interval $(t_1,t_2)$ with $t_1< t_2$. In such a case, we consider the
curve $\tilde{X}$ with same end points $X(t_1)$ and $X(t_2)$ in $J_i$
but linear.  If $p_0 \in \R$ is such that
$p_0e_i=\dot{\tilde{X}}(\tau)$ for $\tau \in (t_1,t_2)$ and $pe_i=
\dot{X}(\tau)$, we deduce from \eqref{eq::r42} that
\begin{equation}\label{eq::r27}
  \int_{t_1}^{t_2} L(X(\tau),\dot{X}(\tau))d\tau \ge \int_{t_1}^{t_2}
  L(\tilde{X}(\tau),\dot{\tilde{X}}(\tau))d\tau + \frac{\gamma}2
  \int_{t_1}^{t_2} |\dot{X}(\tau)-\dot{\tilde{X}}(\tau)|^2d\tau.
\end{equation} \medskip

\noindent \textsc{Case B: $X([t_1,t_2]) \subset J_i$ with
  $X(t_1)=X(t_2)=0$.}
In that case, let us set $\tilde{X}(\tau)=0$ for $\tau\in
[t_1,t_2]$. Using (\ref{eq::r42}) with $p_0=0$ and the definition of
$L_0$ as a minimum of the $L_j$'s (see (\ref{eq::r26})), we get that
$$
L_i(p)\ge L_0(0)+ L_i'(0) p +  \frac{\gamma}2 p^2
$$
from what we deduce that (\ref{eq::r27}) still holds true. \medskip

\noindent \textsc{Case C: the general case.}  By assumption, we have
$X\in {\mathcal A}_0(y;x) \subset C([0,1])$. We then distinguish two
cases.  Either $0\not\in X([0,1])$, and then we define $\tilde{X}$ as
in Case~A. In this case, \eqref{eq::r27} implies that 
$$
\mathcal{E}(X)\ge \dzq (y,x).
$$
Or $0\in X([0,1])$, and then we call $[\tau_1,\tau_2]\subset [s,t]$
the largest interval such that $X(\tau_1)=0=X(\tau_2)$, and define
$\tilde{X}$ as follows: it is linear between $0$ and $\tau_1$, and
reaches $0$ at $\tau_1$; it stays at $0$ in $(\tau_1,\tau_2)$; then it
is linear in $(\tau_2,1)$ and reaches $x$ at $t=1$. Using again the
continuity of $X$, we can find a decomposition of $[\tau_1,\tau_2]$ as
a disjoint union of intervals ${\mathcal I}_k$ (with an at most
countable union)
$$[\tau_1,\tau_2]=\bigcup_{k} {\mathcal I}_k$$ 
such that for each $k$, $X({\mathcal I}_k)\subset J_{i_k}$ for some
$i_k\in I_N$ and $X=0$ on $\partial {\mathcal I}_{k}$.  Using Case A
or Case B on each segment $\overline{{\mathcal I}}_k$, we deduce that
$$
{\mathcal E}(X)\ge \dzun(y,x).
$$
\end{proof}

\subsection{Continuity of the (reduced) minimal action}

\begin{lemma}[Continuity of $\dzun$]\label{lem:dzun-cont}
 The function $\dzun$ is continuous in $J^2$.
\end{lemma}
\begin{proof}
  We first prove that $\dzun$ is lower semi-continuous. We know from
  Lemma~\ref{lem:lsc-ei} that the function
$$
G(\tau_1,\tau_2;y;x)=\mathcal{E}_1(\tau_1,y)+\mathcal{E}_2(\tau_2,x)
$$
is lower semi-continuous for $y,x\in J$ and $0\le \tau_1\le \tau_2\le 1$.
Therefore the function
$$
\dzun(y;x)=\inf_{0\le \tau_1\le \tau_2\le 1}
G(\tau_1,\tau_2;y,x)
$$
is also lower semi-continuous (since the infimum is taken over a
compact set). Besides, the infimum is in fact a minimum.

We now prove that $\dzun$ is upper semi-continuous at any point
$(y,x)$. Consider first $(\tau_1,\tau_2) \in [0,1]^2$ such that
$$
\dzun(y,x) = \mathcal{E}_1(\tau_1,y) + \mathcal{E}_2(\tau_2,x).
$$
Given any sequence $(y^k,x^k)\to (y,x)$, we want to show
that
\begin{equation}\label{eq::r85}
\dzun(y^k,x^k)\le  \dzun(y,x) + o_k(1)
\end{equation}
We use 
$$
\dzun(y^k,x^k)\le \mathcal{E}_1(\tau_1^k,y^k) + \mathcal{E}_2(\tau_2^k,x^k)
$$
with an appropriate choice of $(\tau_1^k,\tau_2^k)$. 

\paragraph{Case 1: $y\in J_j^*$, $x\in J_i^*$.}
In this case, we choose $(\tau_1^k,\tau_2^k) = (\tau_1,\tau_2) \in (0,1)^2$
and we use Lemma~\ref{lem:cont-ei} in order to get
$$
\mathcal{E}_1(\tau_1^k,y^k)\to \mathcal{E}_1(\tau_1,y)
$$
and
$$
\mathcal{E}_2(\tau_2^k,x^k)\to \mathcal{E}_2(\tau_2,x) 
$$
Hence we conclude that (\ref{eq::r85}) holds true.

\paragraph{Case 2: $y=0$, $x\in J_i^*$.}  We choose
$(\tau_1^k,\tau_2^k) = (|y^k|,\max(\tau_2,|y^k|)) \in [0,1)^2.$ We still
have $\tau_2^k\to \tau_2$ and we can use Lemma~\ref{lem:cont-ei} in order
to get
$$
\mathcal{E}_2(\tau_2^k,x^k)\to \mathcal{E}_2(\tau_2,x).
$$
We also have (if $y^k\in J_j$)
\begin{equation}\label{eq::r88}
  \mathcal{E}_1(\tau_1^k,y^k) \le  |y^k| L_j
  \left(-\frac{y^k_j}{|y^k|}\right) 
  -|y^k| L_0 (0)\to 0=\mathcal{E}_1(\tau_1,0).
\end{equation}
Hence we conclude that (\ref{eq::r85}) holds true.

\paragraph{Case 3: $y\in J_j^*$, $x=0$.}  We choose
$(\tau_1^k,\tau_2^k) = (\min(\tau_1,1-|x^k|),1-|x^k|) \in (0,1]^2$ We still
have $\tau_1^k\to \tau_1$ and then
$$
\mathcal{E}_1(\tau_1^k,y^k)\to \mathcal{E}_1(\tau_1,y)
$$
(since $\mathcal{E}_1$ is continuous in $(0,1] \times J^*$).
We also have (if $x^k\in J_i$)
\begin{equation}\label{eq::r89}
  \mathcal{E}_2(\tau_2^k,x^k) \le  |x^k| L_i
  \left(\frac{x^k_i}{|x^k|}\right) 
+(1-|x^k|) L_0 (0)\to L_0(0) = \mathcal{E}_2(\tau_2,0)
\end{equation}
Hence we conclude that (\ref{eq::r85}) holds true.

\paragraph{Case 4: $y=0$, $x=0$.}  We choose $(\tau_1^k,\tau_2^k) =
(|y^k|,1-|x^k|) \in [0,1)\times (0,1]$.  We deduce \eqref{eq::r85} from
\eqref{eq::r88} and \eqref{eq::r89}.
\end{proof}
\begin{lemma}
  The function $\dzo$ is continuous in $J^2$.
\end{lemma}
\begin{proof}
  Since $\dzq$ is lower semi-continuous, we can use
  Lemmas~\ref{lem:optim-curv} and \ref{lem:dzun-cont} in order to conclude
  that $\dzo$ is lower semi-continuous.

Consider $(y,x) \in \partial (J_i \times J_i) \setminus \{(0,0)\}$. Then either $x=0$ or
$y=0$. Moreover for $y=y_ie_i$ and $x=x_ie_i$,
$$
\dzun(y,x)\le \left\{\begin{array}{ll}
    \mathcal{E}_1(1,y)+\mathcal{E}_2(1,x) & \text{ if } x_i=0\\
    \mathcal{E}_1(0,y)+\mathcal{E}_2(0,x) & \text{ if } y_i=0
\end{array}\right\} \le L_i\left(x_i-y_i\right).
$$
Therefore for each $i\in I_N$, we have for $(y,x)\in \partial
(J_i\times J_i)$, 
$$
\dzun(y,x) \le \dzq(y,x) 
$$
Therefore we have with $y=y_ie_i$, $x=x_ie_i$
\begin{equation}\label{eq:dzo-dzun}
\dzo(y,x)=\begin{cases}
  \min(\dzun(y,x), L_i(x_i-y_i)) & \text{ if } (y,x)\in J_i\times J_i\\
  \dzun(y,x) & \text{ if } (y,x)\in \partial (J_i\times
  J_i) \\
  \dzun(y,x) & \text{ otherwise.}
\end{cases}
\end{equation}
This implies that $\dzo$ is continuous in $J^2$.
\end{proof}

\subsection{Study of $\dzun$}

In view of \eqref{eq:dzo-dzun}, we see that the study of $\dzo$
can now be reduced to the study of $\dzun$. The function $\dzun$ is defined as
a minimum over a triangle $\{(\tau_1,\tau_2) \in [0,1]^2: \tau_1 \le \tau_2
\}$. We will see below that $\dzun$ is defined implicitly when the constraint $\tau_1
\le \tau_2$ is active ($\dzt$) or defined explicitly if not ($\dzd$).  
In other words, it will be linear ``as long as'' trajectories stay some
time ($\tau_2 -\tau_1>0$) at the junction point. 

We first define for $(y,x)\in J^2$,
\begin{equation}\label{eq:def-dzt}
\dzt(y,x) = \inf_{0\le \tau\le
  1}\left\{\mathcal{E}_1(\tau,y)+\mathcal{E}_2(\tau,x)\right\}.
\end{equation}
The continuity of $\dzt$ will be used later on. 
\begin{lemma}[Continuity of $\dzt^{ji}$]\label{lem::r101}
  The restrictions $\dzt^{ji}$ of $\dzt$ are continuous in $(J_j\times
  J_i)\setminus \{(0,0)\}$ and continuous at $(0,0)$ if
  $j\in I_0$ or $i\in I_0$.
\end{lemma}
\begin{proof}[Proof of Lemma~\ref{lem::r101}.]
From Lemma~\ref{lem:lsc-ei}, we deduce that $\dzt^{ji}$ is 
lower semi-continuous on $J_j\times J_i$.  We now show that $\dzt^{ji}$ is
upper semi-continuous at any point $(y,x)\in (J_j\times J_i)\setminus
\{(0,0)\}$ and also at $(0,0)$ if $j\in I_0$ or $i\in I_0$. We first
consider $\tau \in [0,1]$ such that
$$
\dzt^{ji}(y,x) = \mathcal{E}_1(\tau,y) + \mathcal{E}_2(\tau,x) \quad
\mbox{with}\quad 0\le \tau\le 1.
$$
For any sequence $(y^k,x^k)\to (y,x)$ with $(y^k,x^k)\in
J_j\times J_i$, we want to show that
\begin{equation}\label{eq::r102}
\dzt^{ji}(y^k,x^k)\le  \dzt^{ji}(y,x) + o_k(1)
\end{equation}
Arguing as in Lemma~\ref{lem:dzun-cont}, we use 
$$
\dzt^{ji}(y^k,x^k)\le \mathcal{E}_1(\tau^k,y^k) + \mathcal{E}_2(\tau^k,x^k)
$$
and we choose $\tau^k$ as follows
$$ 
\begin{array}{ll}
\text{if } y\in J_j^*, x\in J_i^*: &  \tau^k=\tau \in (0,1), \\
\text{if } y=0, x\in J_i^*: & \tau^k = |y_k| \in [0,1), \\
\text{if } y\in J_j^*, x=0: & \tau^k=1-|x^k| \in (0,1], \\
\text{if } x=0, j\in I_0: & \tau^k=1-|x^k| \in (0,1], \\
\text{if } y=0, x=0, i\in I_0: & \tau^k=|y^k| \in [0,1).
\end{array}
$$
This ends the proof of the lemma.
\end{proof}
We next define for $(y,x) \in J^j \times J^i$ 
\begin{equation}\label{eq:def-dzd}
\dzd^{ji}(y,x)=- L'_j (\xi^-_j) y + L'_i (\xi^+_i) x + L_0 (0)
\end{equation}
where $\xi^\pm_l$ are defined thanks to the following function (for $l \in
I_N$) 
$$
K_l (\xi) = L_l (\xi) - \xi L'_l (\xi) - L_0 (0).
$$ 
Precisely, $\xi^\pm_l = (K_l^\pm)^{-1}(0) \neq 0$ when $l \notin I_0$ (see
Lemma~\ref{lem:k} below).  We will see that $K_l$ plays an important role
in the analysis of $\dzun$.  In particular, it allows us to define, when $i
\notin I_0$ and $j \notin I_0$, the following convex subset (triangle) of
$J_j \times J_i$:
$$
\Delta^{ji}=\left\{(y,x)\in J_j\times J_i, \quad \frac{x}{\xi^+_i} -
  \frac{y}{\xi^-_j} <1\right\}
$$
It is convenient to set $\Delta^{ji}=\emptyset$ if $i \in I_0$ or $j
\in I_0$. 
We next state a series of lemmas before proving them. 
\begin{lemma}[Link between $\dzun,\dzd,\dzt$]\label{lemma:junc-1}
\begin{equation}\label{eq:dzun-dzd-dzt}
\dzun^{ji}(y,x)=\left\{\begin{array}{ll}
    \dzd^{ji}(y,x) & \quad \mbox{if}\quad (y,x)\in \Delta^{ji}\\
    \dzt^{ji}(y,x) & \quad \mbox{if}\quad (y,x)\in (J_j\times
    J_i)\setminus \Delta^{ji}.
\end{array}\right.
\end{equation}
\end{lemma}
\begin{lemma}[The equations in the interior]\label{lemma:junc-2}
The functions $\dzun^{ji}$, $\dzd^{ji}$ and $\dzt^{ji}$ are
convex and $C^1$ in $J_j^*\times J_i^*$  and, if
$\tilde{\dz}$ is one of them, it satisfies for $(y,x)\in J_j^*\times
J_i^*$
\begin{equation}\label{eq::r99}
  \left\{\begin{array}{l}
      \tilde{\dz}(y,x)-x\partial_x \tilde{\dz}(y,x) - y\partial_y \tilde{\dz}(y,x) + H_i(\partial_x \tilde{\dz}(y,x))=0,\\
      \tilde{\dz}(y,x)-x\partial_x \tilde{\dz}(y,x) - y\partial_y \tilde{\dz}(y,x) + H_j(-\partial_y \tilde{\dz}(y,x))=0.
\end{array}\right.
\end{equation}
\end{lemma}
\begin{lemma}[Study of $\dzt$]\label{lemma:junc-3}
For $(y,x)\in J_j^*\times J_i^*$, there exists a unique $\tau=T(y,x)\in (0,1)$ such that
$$
\dzt^{ji}(y,x)=\mathcal{E}_1(\tau,y)+\mathcal{E}_2(\tau,x);
$$
Moreover,
$$
\left\{\begin{array}{ll}  \partial_x \dzt^{ji}(y,x)
    =L_i'\left(\xi_x\right)
    &\quad \mbox{with}\quad  \xi_x = \frac{x}{1-T(y,x)},\\
     \partial_y \dzt^{ji}(y,x)
    =-L_j'\left(\xi_y\right)&\quad \mbox{with}\quad 
    \xi_y = -\frac{y}{T(y,x)}.
\end{array}\right.
$$ 
\end{lemma}
\begin{lemma}[Study of $T$]\label{lem::r111}
  For $(y,x)\in (J_j\times J_i)\setminus \left\{(0,0)\right\}$, there is a
  unique $\tau=T(y,x)\in [0,1]$ such that
$$
\dzt^{ji}(y,x)=\mathcal{E}_1(\tau,y)+\mathcal{E}_2(\tau,x).
$$
Moreover  $T\in C(J_j\times J_i\setminus \left\{(0,0)\right\})$ and 
$$
T(y,x)= 
\begin{cases}
  \max\left(0,1-\frac{x}{\xi_i^+}\right) & \mbox{if}\quad
    (y,x)\in (\left\{0\right\} \times J_i^*) \setminus
    \Delta^{ji}, \\
    \min\left(1,-\frac{y}{\xi_j^-}\right) & 
  \mbox{if}\quad (y,x)\in (J_j^* \times \left\{0\right\})\setminus
  \Delta^{ji}.
\end{cases}
$$
\end{lemma}
\begin{lemma}[$\dzun^{ji}$ at the boundary]\label{lem::r91}
  Then we have $\dzun^{ji}\in C^1(J_j\times J_i)$ with
\begin{equation}\label{eq::r106}
\left\{\begin{array}{l}
 \partial_x \dzun^{ji}(y,x) =L_i'\left(\xi_x\right)\\
  \partial_y \dzun^{ji}(y,x) =-L_j'\left(\xi_y\right)
\end{array}\right.
\end{equation}
where $\xi_y\le 0\le \xi_x$ satisfy
\begin{equation}\label{eq::r109}
(\xi_x,\xi_y)=  \begin{cases}
      (\max(x,\xi_i^+), (K_j^-)^{-1}(K_i(\xi_x))) &
       \text{ if } (y,x)\in (\left\{0\right\} \times J_i) \setminus \Delta^{ji}\\
      (\xi_i^+, \xi_j^-) &
       \text{ if } (y,x)\in (\left\{0\right\} \times J_i) \cap \Delta^{ji}\\
      ((K_i^+)^{-1}(K_j(\xi_y)), -\max(y,-\xi_j^-)) 
      & \text{ if } (y,x)\in  (J_j \times \left\{0\right\})\setminus \Delta^{ji}\\
      ( \xi_i^+, \xi_j^-)
      & \text{ if } (y,x)\in  (J_j \times \left\{0\right\})\cap  \Delta^{ji}.
\end{cases}
\end{equation}
Moreover we have
\begin{equation}\label{eq::r116bis}
\left\{\begin{array}{ll}
 \dzun^{ji}(0,x)=\frac{x}{\xi_x}\left(L_i(\xi_x)-L_0(0)\right)+L_0(0) & \quad \mbox{for}\quad x\in J_i^*\\
 \dzun^{ji}(y,0)=-\frac{y}{\xi_y}\left(L_j(\xi_y)-L_0(0)\right)+L_0(0) & \quad \mbox{for}\quad y\in J_j^*
\end{array}\right.
\end{equation}
and 
\begin{multline}\label{eq::r115}
\dzun^{ji}(x,y)-x\partial_x\dzun^{ji}(x,y)-y\partial_y\dzun^{ji}(x,y)  
\\=\left\{\begin{array}{ll}
L_0(0)+ K_i(\max(x,\xi_i^+)) &\quad \mbox{if}\quad (y,x)\in \left\{0\right\}\times J_i,\\
L_0(0)+ K_j(-\max(y,-\xi_j^-)) &\quad \mbox{if}\quad (y,x)\in  J_j\times\left\{0\right\}.
\end{array}\right.
\end{multline}
\end{lemma}
Before proving these lemmas, the reader can check the following useful
properties of the function $K_l$ that will be used in their proofs.
\begin{lemma}[Properties of $K_l$]\label{lem:k}
  Assume (A1). Then for any $l\in I_N$, we have
\begin{align*}
K_l'(\xi)\ge \gamma |\xi|\quad &\mbox{for}\quad \xi\in (-\infty,0),\\
K_l'(\xi) \le -\gamma |\xi|\quad &\mbox{for}\quad \xi\in (0,+\infty).
\end{align*}
We define $(K_l^-)^{-1}$ as the inverse of the function $K_l$ restricted to $(-\infty,0]$,
and $(K_l^+)^{-1}$ as the inverse of the function $K_l$ restricted to $[0,+\infty)$.
We set
$$
\xi_l^\pm = (K_l^\pm)^{-1}(0).
$$
Then we have
\begin{align*}
\pm \xi_l^\pm =0 \quad &\mbox{if}\quad l\in I_0,\\
\pm \xi_l^\pm >0 \quad &\mbox{if}\quad l\in I_N\setminus I_0.
\end{align*}
Moreover we have
\begin{equation}\label{eq::r98}
K_l (\xi)=-H_l (L'_l (\xi)) -L_0 (0).
\end{equation}
\end{lemma}
\begin{proof}[Proof of Lemmas~\ref{lemma:junc-1}-\ref{lemma:junc-3}.]
The proof proceeds in several steps. \medskip

\noindent \textbf{Step 1: first study of $\dzun^{ji}$.}
Let us define
$$
G(\tau_1,\tau_2,y,x)=\mathcal{E}_1(\tau_1,y)+\mathcal{E}_2(\tau_2,x).
$$
For $\tau_1,\tau_2\in (0,1)$, and setting
\begin{equation}\label{eq::r97}
\xi_y = -\frac{y}\tau_1,\quad \xi_x = \frac{x}{1-\tau_2}
\end{equation}
and $V_y=(\xi_y,0,1,0)$ and $V_x=(0,\xi_x,0,1)$, we compute
$$
D^2 G (\tau_1,\tau_2,y,x)= \frac{L_j''(\xi_y)}{\tau_1}V_y^T V_y +
\frac{L_i''(\xi_x)}{1-\tau_2}V_x^T V_x \ge 0.
$$
Therefore $G$ is in particular convex on $(0,1)\times (0,1)\times
J_j^*\times J_i^*$.  Because $G$ is in particular lower
semi-continuous on $[0,1]\times [0,1]\times J_j^*\times J_i^*$, and
\begin{equation}\label{eq::r96}
  G(0,\tau_2,y,x)=+\infty=G(\tau_1,1,y,x)\quad \text{ for } (y,x)\in J_j^*\times J_i^*,
\end{equation} 
we deduce that
$$
\dzun^{ji}(y,x)=\inf_{0< \tau_1\le \tau_2< 1}
G(\tau_1,\tau_2,y,x)\quad \mbox{for}\quad (y,x)\in J_j^*\times J_i^*.
$$
This implies that $\dzun^{ji}$ is also convex in $J_j^*\times J_i^*$.
Notice that in particular
$$
D^2_{\tau_1\tau_1}G(\tau_1,\tau_2,y,x) = \frac{y^2}{\tau_1^3} L''_j
(\xi_y)>0
$$
and
$$
D^2_{\tau_2\tau_2}G(\tau_1,\tau_2,y,x) =  \frac{x^2}{(1-\tau_2)^3}
L''_i (\xi_x)>0.
$$
The map $(\tau_1,\tau_2)\mapsto G(\tau_1,\tau_2,y,x)$ is then
strictly convex on the convex set 
$$
\left\{(\tau_1,\tau_2)\in  (0,1)^2,\quad \tau_1\le \tau_2\right\}.
$$ 
Therefore using again (\ref{eq::r96}) and the lower semi-continuity of
$G$, we deduce that it has a unique minimum that we denote by
$(\tau_1,\tau_2)$ satisfying $0<\tau_1\le \tau_2< 1$.  \medskip

\noindent \textbf{Step 2: study of $\dzt^{ji}$.} Let us consider the
following function
$$
e(\tau,y,x)=G(\tau,\tau,y,x).
$$
For $\tau\in (0,1)$, setting 
$$
\xi_y = -\frac{y}\tau,\quad \xi_x = \frac{x}{1-\tau}
$$
and proceeding similarly as in Step 1, we can deduce that
$$
\dzt^{ji}(y,x)=\inf_{\tau\in (0,1)} e(\tau,y,x) \quad \mbox{for}\quad
(y,x)\in J_j^*\times J_i^*
$$
and that $\dzt^{ji}$ is also convex on $J_j^*\times J_i^*$.  We can
also deduce that the map $\tau \mapsto e(\tau,y,x)$ is strictly convex
on $(0,1)$ for $(y,x)\in J_j^*\times J_i^*$ and that it has a unique
minimum that we denote by $\tau\in (0,1)$ such that
$$
\dzt^{ji}(y,x)=  e(\tau,y,x).
$$

Using the derivative with respect to $\tau$, we see that $\tau$ is characterized by the equation 
\begin{equation}\label{eq:cno}            
F=0 \quad \mbox{with}\quad F(\tau,y,x):= K_j (-\frac{y}\tau) - K_i (\frac{x}{1-\tau}).
\end{equation} 
Moreover
$$
\partial_\tau F(\tau,y,x)=D^2_{\tau\tau}e(\tau,y,x)>0.
$$
Using the regularity $C^2$ of $L_l$ given in assumption (A1), we see
that the unique solution $\tau = T(y,x)$ of $F(\tau,y,x)=0$ is
continuously differentiable with respect to $(y,x)$. Therefore we
deduce that $\dzt^{ji}\in C^1(J_j^*\times J_i^*)$.

We have 
\begin{eqnarray}
\dzt^{ji} (y,x) &=& \mathcal{E}_1 (T(y,x),y) + \mathcal{E}_2 (T(y,x),x), \label{eq::r31} \\
\partial_y \dzt^{ji} (y,x) &=& (\partial_y \mathcal{E}_1) (T(y,x),y) = - L'_j
(\xi_y), \label{eq::r32} \\
\partial_x \dzt^{ji} (y,x) &=& (\partial_x \mathcal{E}_2) (T(y,x),x) = L'_i
(\xi_x). \label{eq:dzo-dir}
\end{eqnarray}
Writing $\tau$ for $T(y,x)$, and using the optimality
condition~\eqref{eq:cno}, we get
\begin{align*}
(\dzt^{ji} &- x \partial_x \dzt^{ji}  - y\partial_y \dzt^{ji} )(y,x) \\
&= \tau K_j (-\frac{y}\tau) + (1-\tau) K_i (\frac{x}{1-\tau}) +L_0(0)\\
&= K_j (-\frac{y}\tau) +L_0(0)= -H_j (L'_j (-\frac{y}\tau)) \\ 
&= -H_j(-\partial_y \dzt^{ji} (y,x)) \\
&= K_i (\frac{x}{1-\tau}) +L_0(0)= -H_i (L'_i (\frac{x}{1-\tau})) \\ 
&= -H_i
(\partial_x \dzt^{ji} (y,x)) 
\end{align*}
where we have used (\ref{eq::r98}) in the second and in the fourth
line.  Hence $\dzt^{ji}$ satisfies (\ref{eq::r99}) on $J_j^*\times
J_i^*$. \medskip

\noindent \textbf{Step 3: further study of $\dzun^{ji}$.}
We concluded at the end of Step 1 that for $(y,x)\in J_j^*\times J_i^*$ we have
$$
\dzun^{ji}(y,x)=\mathcal{E}_1(\tau_1,y) + \mathcal{E}_2(\tau_2,x)
$$
with $0<\tau_1\le \tau_2< 1$. Then we can distinguish two cases.
\medskip

\noindent \textbf{Case 1: $\tau_1<\tau_2$.} In that case this implies that
$$
\partial_{\tau_1} \mathcal{E}_1 (\tau_1,y) = 0, \quad \partial_{\tau_2} \mathcal{E}_2
(\tau_2,x) = 0
$$
which can be written as
\begin{equation}\label{racinedek}
  K_j (\xi_y) = 0, \quad K_i (\xi_x) = 0 
\end{equation}
with $\xi_y\le 0\le \xi_x$ defined in (\ref{eq::r97}).

Using Lemma~\ref{lem:k}, we conclude that \eqref{racinedek} holds true
if and only if $K_j(0) >0$ and $K_i(0)>0$; i.e. $j,i\in I_N\setminus
I_0$.  In this case we have $\xi_y = \xi^-_j$ and $\xi_x=\xi^+_i$ and
then
\begin{equation}\label{eq::r100}
  \tau_1 = -\frac{y}{\xi^-_j}, \quad \tau_2 = 1 - \frac{x}{\xi^+_i}.
\end{equation} 
Moreover, we have in this case  $\dzun^{ji} (y,x) = \dzd^{ji}(y,x)$.

Using Legendre-Fenchel's equality together with $K_j (\xi^-_j)=0$ and
$K_i (\xi^+_i)=0$, we have
\begin{equation}\label{eq::r130}
  \dzd^{ji} (y,x) - y \partial_y \dzd^{ji} (y,x) - x \partial_x \dzd^{ji} (y,x) = L_0 (0),
\end{equation}
and
\begin{align*}
H_i (\partial_x \dzd^{ji} (y,x)) = H_i (L'_i (\xi^+_i))
 & = \xi^+_i L'_i (\xi^+_i) - L_i (\xi^+_i ) = -L_0 (0),\\
H_j (-\partial_y \dzd^{ji} (y,x)) = H_j (L'_j (\xi^-_j))
&= \xi^-_j L'_j (\xi^-_j) - L_j (\xi^-_j) = -L_0 (0).
\end{align*}
Hence $\dzd^{ji}$ satisfies (\ref{eq::r99}) on $J_j^*\times J_i^*$.

Finally we deduce from \eqref{eq::r100} that the condition: $0 <
\tau_1 < \tau_2 <1$ is equivalent to $(y,x)\in \Delta^{ji}\cap
(J^*)^2$ and then by continuity of $\dzun^{ji}$ and $\dzd^{ji}$, we
get
$$
\dzun^{ji}= \dzd^{ji} \quad \mbox{on}\quad  \Delta^{ji}.
$$  

\noindent \textbf{Case 2: $\tau_1=\tau_2$.} If for $(y,x)\in J_j^*\times
J_i^*$ we have
$$
\dzun^{ji}(y,x)=\mathcal{E}_1(\tau_1,y)+\mathcal{E}_2(\tau_2,x)
$$
with $\tau_1=\tau_2$, then we have seen that $(y,x)\in (J_j^*\times
J_i^*)\setminus \Delta^{ji}$ and $\dzun^{ji}(y,x)=\dzt^{ji}(y,x)$.
From Lemma~\ref{lem::r101}, we also have that $\dzt^{ji}\in
C(J_j\times J_i)$ if $j\in I_0$ or $i\in I_0$ and in that case
$\Delta^{ji}=\emptyset$. On the other hand, we have $\dzt^{ji}\in
C((J_j\times J_i)\setminus \{(0,0)\})$ if $j,i\in
I_N\setminus I_0$ with $\{(0,0)\}\in \Delta^{ij}$ in that
case.  Therefore in all cases we have
$$
\dzt^{ji}\in C((J_j\times J_i)\setminus \Delta^{ji}).
$$
Now from the continuity of $\dzun$, we deduce that
$$
\dzun^{ji}=  \dzt^{ji} \quad \mbox{on}\quad  (J_j\times J_i)\setminus
\Delta^{ji}.
$$

\noindent \textbf{Step 4: on the boundary $(\partial \Delta^{ji})\cap
  (J^*)^2$.}
We already know that $\dzun$ is continuous, therefore if $j,i\in I_N\setminus I_0$:
$$
\dzd^{ji}=\dzt^{ji} \quad \mbox{on}\quad \left\{(y,x)\in J_j\times
  J_i, \quad \frac{x}{\xi^+_i} - \frac{y}{\xi^-_j} =1\right\}
$$
On the other hand, recall that $(y,x)\in J_j^*\times J_i^*$, the real
$\tau\in (0,1)$ is characterized by (\ref{eq:cno}), i.e.
\begin{equation}\label{eq::r103}
  K_j\left(\xi_y\right)=K_i\left(\xi_x\right) \quad 
\mbox{with}\quad \xi_y=-\frac{y}{\tau},\quad \xi_x=\frac{x}{1-\tau}.
\end{equation}
Notice that if we choose
$$\tau=-\frac{y}{\xi_j^-}$$
we deduce from $\frac{x}{\xi^+_i} - \frac{y}{\xi^-_j} =1$ that
\begin{equation}\label{eq::r104}
\xi_y=\xi_j^- \quad \mbox{and}\quad \xi_x=\xi_i^+
\end{equation}
which are obvious solutions of (\ref{eq::r103}).  Therefore we
conclude that this is the solution. Using
(\ref{eq::r32})-(\ref{eq:dzo-dir}) and the expression 
of $\dzd^{ji}$, (\ref{eq::r104}) implies the equality of the gradients
of $\dzd^{ji}$ and $\dzt^{ji}$ on the boundary $(\partial
\Delta^{ji})\cap (J^*)^2$.  Finally this shows that $\dzun^{ji}\in
C^1(J_j^*\times J_i^*)$.  This ends the proof of the lemmas.
\end{proof}
\begin{proof}[Proof of Lemma~\ref{lem::r111}.]
The proof proceeds in several steps. \medskip

 \noindent \textbf{Continuity of $T$.} We set for $(\tau,y,x)\in
  [0,1]\times J_j\times J_i$
$$
e(\tau,y,x)=\mathcal{E}_1(\tau,y)+\mathcal{E}_2(\tau,x).
$$
From Proposition~\ref{lemma:junc-1}, we already know that there exists a
unique $\tau\in [0,1]$ such that
$$
\dzt^{ji}(y,x)=e(\tau,y,x) \quad \mbox{if}\quad (y,x)\in J_j^*\times
J_i^*.
$$
On the other hand, we have
\begin{equation}\label{eq::r116}
  e(\tau,y,x)= \left\{\begin{array}{lll}
       (1-\tau)L_i\left(\frac{x}{1-\tau}\right)+\tau L_0(0) 
      & \text{ if }  (y,x)\in \left\{0\right\}\times J_i^*&  \mbox{(case 1),}\\
       \tau L_j\left(-\frac{y}{\tau}\right)+(1-\tau) L_0(0)
      &\text{ if } (y,x)\in J_j^*\times\left\{0\right\}&  \mbox{(case 2).}
\end{array}\right.
\end{equation}
Notice that in Cases~1 and 2, there is a unique $\tau\in [0,1]$ 
such that
\begin{equation}\label{eq::r112}
\dzt^{ji}(y,x)=e(\tau,y,x)
\end{equation}
and $\tau\in [0,1)$ in case 1, $\tau\in (0,1]$ in case 2.  Then the
continuity of $\tau=T(y,x)$ in $(J_j\times J_i)\setminus
\{(0,0)\}$ follows from the lower semi-continuity of $e$ on
$[0,1]\times J_j\times J_i$ and the uniqueness of $\tau$ such that
(\ref{eq::r112}) holds.
\medskip

\noindent \textbf{Computation of $T$.} We distinguish cases. \medskip

\noindent \textsc{Case 1:} $(y,x)\in (\left\{0\right\} \times J_i^*) \setminus \Delta^{ji}$.
Notice that we have
$$
\partial_\tau e(\tau,0,x)=-K_i(\xi_x) \quad \mbox{with}\quad
\xi_x=\frac{x}{1-\tau}.
$$
If $x\ge \xi_i^+$, then $\partial_\tau e(\tau,0,x) \ge 0$ and $T(0,x)=0$.

If $x< \xi_i^+$, then $\xi_x=\xi_i^+$ is a solution of $\partial_\tau
e(\tau,0,x) = -K_i(\xi_x)=0$ and $T(0,x)=1-\frac{x}{\xi_i^+}$.

\noindent \textsc{Case 2:} $(y,x)\in  (J_j^* \times
  \left\{0\right\})\setminus \Delta^{ji}$.
Notice that we have
$$
\partial_\tau e(\tau,y,0)=K_j(\xi_y) \quad \mbox{with}\quad
\xi_y=-\frac{y}{\tau}.
$$
If $y\ge -\xi_j^-$, then $\partial_\tau e(\tau,y,0)\le 0$ and $T(y,0)=1$.

If $y< -\xi_j^-$, then $\xi_y=\xi_j^-$ is a solution of $\partial_\tau
e(\tau,y,0) = K_j(\xi_y)=0$ and $T(y,0)=-\frac{y}{\xi_j^-}$.  This ends the
proof of the lemma.
\end{proof}
\begin{proof}[Proof of Lemma~\ref{lem::r91}.]
The proof proceeds in several steps. \medskip

  \noindent \textbf{Step 1: continuity.} From Proposition~\ref{lemma:junc-1}, we
  already know that $\dzun^{ji}\in C^1((J_j^* \times J_i^*)\cup
  \Delta^{ji})$ and (\ref{eq::r106}) holds true with
$$
\left\{
  \begin{array}{lll}
     \xi_x= \frac{x}{1-\tau},& \quad \xi_y = 
    -\frac{y}{\tau} &\quad \mbox{if}\quad (y,x)\in (J_j^* \times J_i^*) \setminus \Delta^{ji}\\
     \xi_x= \xi_i^+,& 
    \quad \xi_y = \xi_j^- &\quad \mbox{if}\quad (y,x)\in \Delta^{ji}
  \end{array}
\right.
$$
where $\tau=T(y,x)$ in the first line.  Therefore, in order to prove
that $\dzun^{ji}\in C^1(J_j \times J_i)$, it is sufficient to prove
that if $(y,x)\in (\partial (J_j \times J_i))\setminus \Delta^{ji} =
((\left\{0\right\}\times J_i)\times (J_j\times
\left\{0\right\}))\setminus \Delta^{ji}$, and if $(y^k,x^k)\in (J_j^*
\times J_i^*) \setminus \Delta^{ji}$ is a sequence of points such
that $(y^k,x^k)\to (y,x)$, then we have with $\tau^k=T(y^k,x^k)$
\begin{equation}\label{eq::r108}
-\frac{y^k}{\tau^k}\to \xi_y \quad \mbox{and}\quad \frac{x^k}{1-\tau^k}\to \xi_x
\end{equation}
where $(\xi_y,\xi_x)$ is given by (\ref{eq::r109}).  Let us recall
that $\tau^k$ is characterized by the equation
\begin{equation}\label{eq::r107}
   K_j\left(-\frac{y^k}{\tau^k}\right)=K_i\left(\frac{x^k}{1-\tau^k}\right)
\end{equation}
We will assume (up to extract a subsequence) that $\tau^k\to \tau_0$
for some limit $\tau_0\in [0,1]$.  Because we have $|x^k|^2+|y^k|^2\le
R^2$, it is easy to deduce from (\ref{eq::r107}), that there exists a
constant $C_R$ such that
\begin{equation}\label{eq::r110}
\left|-\frac{y^k}{\tau^k}\right| + \left|\frac{x^k}{1-\tau^k}\right|\le C_R
\end{equation}
This can be proved by contradiction, distinguishing the cases
$\tau_0=0$, $\tau_0=1$ and $\tau_0\in (0,1)$.  Up to extract a
subsequence, we can then pass to the limit in (\ref{eq::r107}) and get
\begin{equation}\label{eq::r113}
   K_j\left(\xi_y\right)=K_i\left(\xi_x\right) 
  \quad \mbox{with}\quad \xi_y\le 0\le \xi_x
\end{equation}
In the following cases, we now identify one of the two quantities
$\xi_y$ or $\xi_x$, the other one being determined by
(\ref{eq::r113}).

\noindent \textbf{Case 1: $(y,x)\in (\left\{0\right\}\times
  J_i^*)\setminus \Delta^{ji}$.}  From Lemma~\ref{lem::r111}, we know
that $\tau_0=\max\left(0,1-\frac{x}{\xi_i^+}\right)$, and then
$$
\xi_x=\max (x,\xi_i^+),\quad \xi_y = (K_j^-)^{-1}(K_i(\xi_x))
$$
and from (\ref{eq::r116}), we get
\begin{equation}\label{eq::r118}
\dzun^{ji}(0,x)=\frac{x}{\xi_x}\left(L_i(\xi_x)-L_0(0)\right)+L_0(0)
\end{equation}

\noindent \textbf{Case 2: $(y,x)\in (J_j^*\times
  \left\{0\right\})\setminus \Delta^{ji}$.} From
Lemma~\ref{lem::r111}, we know that
$\tau_0=\min\left(1,-\frac{y}{\xi_j^-}\right)$, and then
$$
-\xi_y=\max ( y, -\xi_j^-),\quad \xi_x = (K_i^+)^{-1}(K_j(\xi_y))
$$
and from (\ref{eq::r116}), we get
\begin{equation}\label{eq::r119}
\dzun^{ji}(y,0)=-\frac{y}{\xi_y}\left(L_j(\xi_y)-L_0(0)\right)+L_0(0)
\end{equation}

\noindent \textbf{Case 3: $(y,x)\in \{(0,0)\} \setminus
  \Delta^{ji}$.} This case only occurs if $j\in I_0$ or $i\in I_0$.
Moreover at least one of the two quantities $-\frac{y^k}{\tau^k}$ and
$\frac{x^k}{1-\tau^k}$ tends to zero.

If $\xi_y=0$, then $K_i(\xi_x)=K_j(0)$ and hence
$$
\xi_y=0\quad \Longrightarrow \quad L_i(0)\ge L_j(0)=L_0(0)
$$
If $\xi_x=0$, then  $K_j(\xi_y)=K_i(0)$
and hence
$$
\xi_x=0\quad \Longrightarrow \quad L_j(0)\ge L_i(0)=L_0(0)
$$
This implies that
$$
\left\{\begin{array}{lll}
\xi_x= \xi_i^+=0, &\quad \xi_y=\xi_j^- <0, &\quad \mbox{if}\quad L_i(0)=L_0(0)<L_j(0),\\
\xi_x=\xi_i^+>0, &\quad \xi_y= \xi_j^-=0,& \quad \mbox{if}\quad L_i(0)>L_j(0)=L_0(0),\\
\xi_x=\xi_i^+=0, &\quad \xi_y= \xi_j^-= 0,& \quad \mbox{if}\quad L_i(0)=L_j(0)=L_0(0).
\end{array}\right.
$$
By the uniqueness of the limit, this finally shows that $\dzun^{ji}\in
C^1(J_j\times J_i)$ and (\ref{eq::r109}) holds. \medskip

\noindent \textbf{Step 2: checking (\ref{eq::r116}) and (\ref{eq::r115}).}
From (\ref{eq::r118}) and (\ref{eq::r119}), we deduce (\ref{eq::r116})
on $((J_j^*\times \left\{0\right\})\cup (\left\{0\right\}\times
J_i^*))\setminus \Delta^{ji}$. From $\dzun^{ji}=\dzd^{ji}$ on
$\Delta^{ji}$, we deduce that (\ref{eq::r116}) is also true on
$((J_j^*\times \left\{0\right\})\cup (\left\{0\right\}\times
J_i^*))\cap \Delta^{ji}$.

Then (\ref{eq::r115}) follows from a simple computation for
$(y,x)\not= (0,0)$.  This is still true for $(y,x)=0$, because
$\dzun^{ji}$ is $C^1$.  This ends the proof of the lemma.
\end{proof}

\subsection{Study of $\dzq$}

The following lemma will be used below. Since it is elementary, its
proof is omitted. 
\begin{lemma}[Properties of $\dzq^{ji}$]\label{lem::r123}
  For $j=i\in I_N$, we have for $(y,x)\in J_j\times J_i$ with
  $(y,x)\not= (0,0)$ if $j=i\in I_N\setminus I_0$:
\begin{align*}
  \dzq^{ji}(y,x)-x\partial_x\dzq^{ji}(y,x)-y\partial_y\dzq^{ji}(y,x) &= 
  L_0(0)+K_i(x-y) \\ &= -H_i(\partial_x\dzq^{ji}(y,x)) \\
&  -H_j(-\partial_y\dzq^{ji}(y,x))
\end{align*}
and
$$
\partial_x\dzq^{ji}(y,x) =L_i'(x-y),\quad \partial_y\dzq^{ji}(y,x)
=-L_j'(x-y).
$$
\end{lemma}

\subsection{Proof of Theorem~\ref{thm:minimal-action-0}}

We are now in position to prove Theorem~\ref{thm:minimal-action-0}. 
We prove several lemmas successively. 
\begin{lemma}[Properties of $\dzo^{ji}$]\label{lem:dzo-1}
  For $(y,x)\in J_j\times J_i$, we have
$$
\dzo^{ji}(y,x)=\left\{\begin{array}{ll}
L_i(x-y)  & \quad \mbox{if}\quad i=j\in I_0,\\
\dzun^{ji}(y,x) & \quad \mbox{if}\quad i\not=j,\\
\min(\dzun^{ji}(y,x),L_i(x-y)) & \quad \mbox{if}\quad i=j\in I_N\setminus I_0.
\end{array}\right.
$$
In particular $\dzo^{ji}\in C^1(J_j\times J_i)$ in the first two cases.
\end{lemma}
\begin{lemma}[Singularities of the gradient of $\dzo$]\label{lem:dzo-1bis}
In the case $i=j\in I_N\setminus I_0$, we have
\begin{equation}\label{eq::r120}
  \dzo^{ji}(y,x) = \left\{\begin{array}{ll}
      \dzd^{ji} & \mbox{ in a neighborhood of }
       (\partial (J_j\times J_i))\cap \Delta^{ji},\\
      L_i(x-y) & \mbox{ in a neighborhood of }
       (\partial (J_j\times J_i))\setminus \overline{\Delta^{ji}};
\end{array}\right.
\end{equation}
moreover, in this case there exists a curve $\Gamma^{ji}$ such that
$\dzo^{ji}\in C^1((J_j\times J_i)\setminus (\Gamma^{ji}\cup
\left\{Y_j,X_i\right\}))$. This curve connects $Y_j= (-\xi_j^-,0)$ and
$X_i=(0,\xi_i^+)$ and is contained in $(J_j^*\times J_i^*)\cap \Delta^{ji}$
\end{lemma}
The results of these two lemmas are illustrated in
Figures~ \ref{fig2} and \ref{fig4}.
\begin{figure}[h]
\begin{center}
\resizebox{4cm}{!}{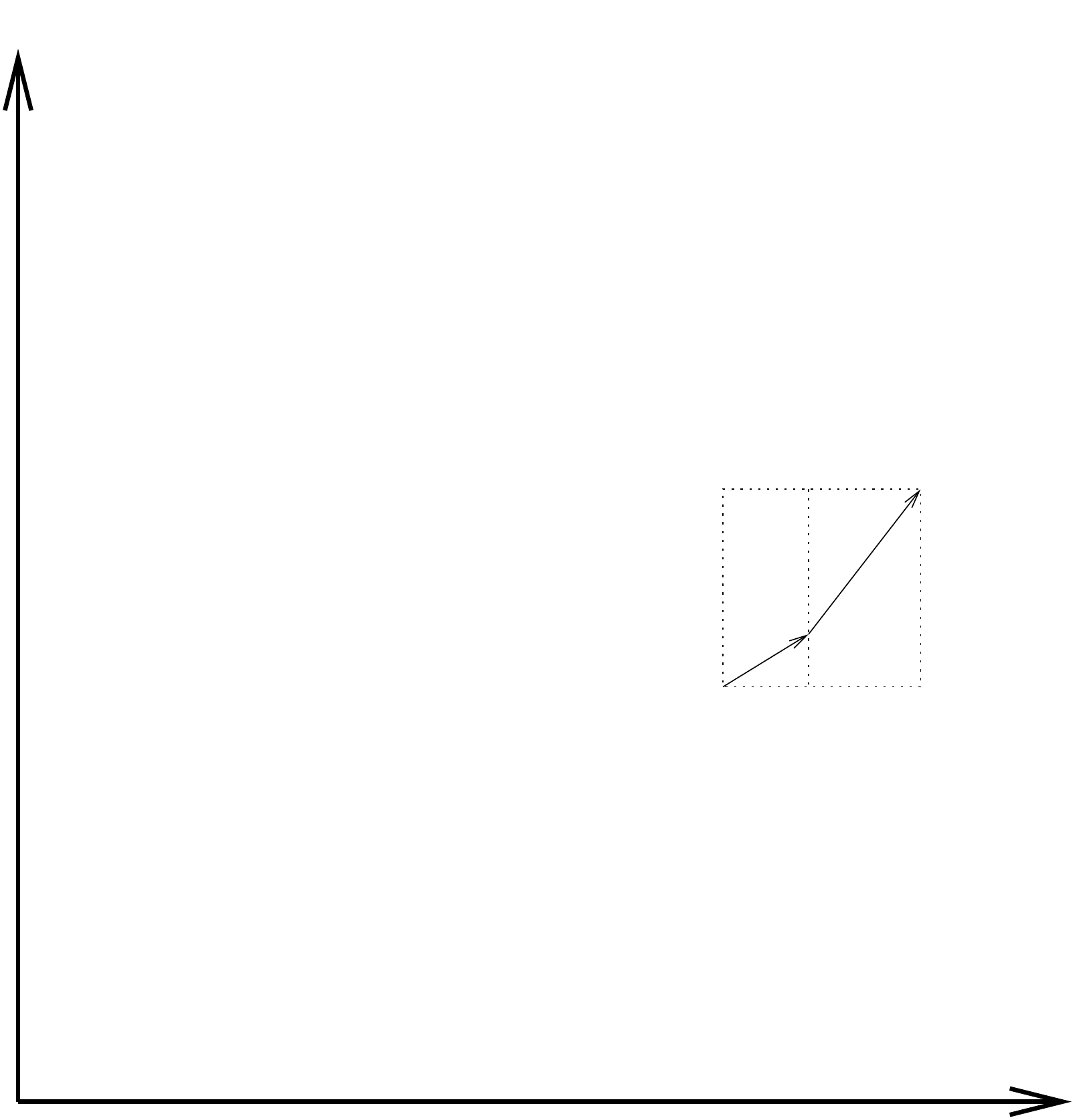}
\resizebox{4cm}{!}{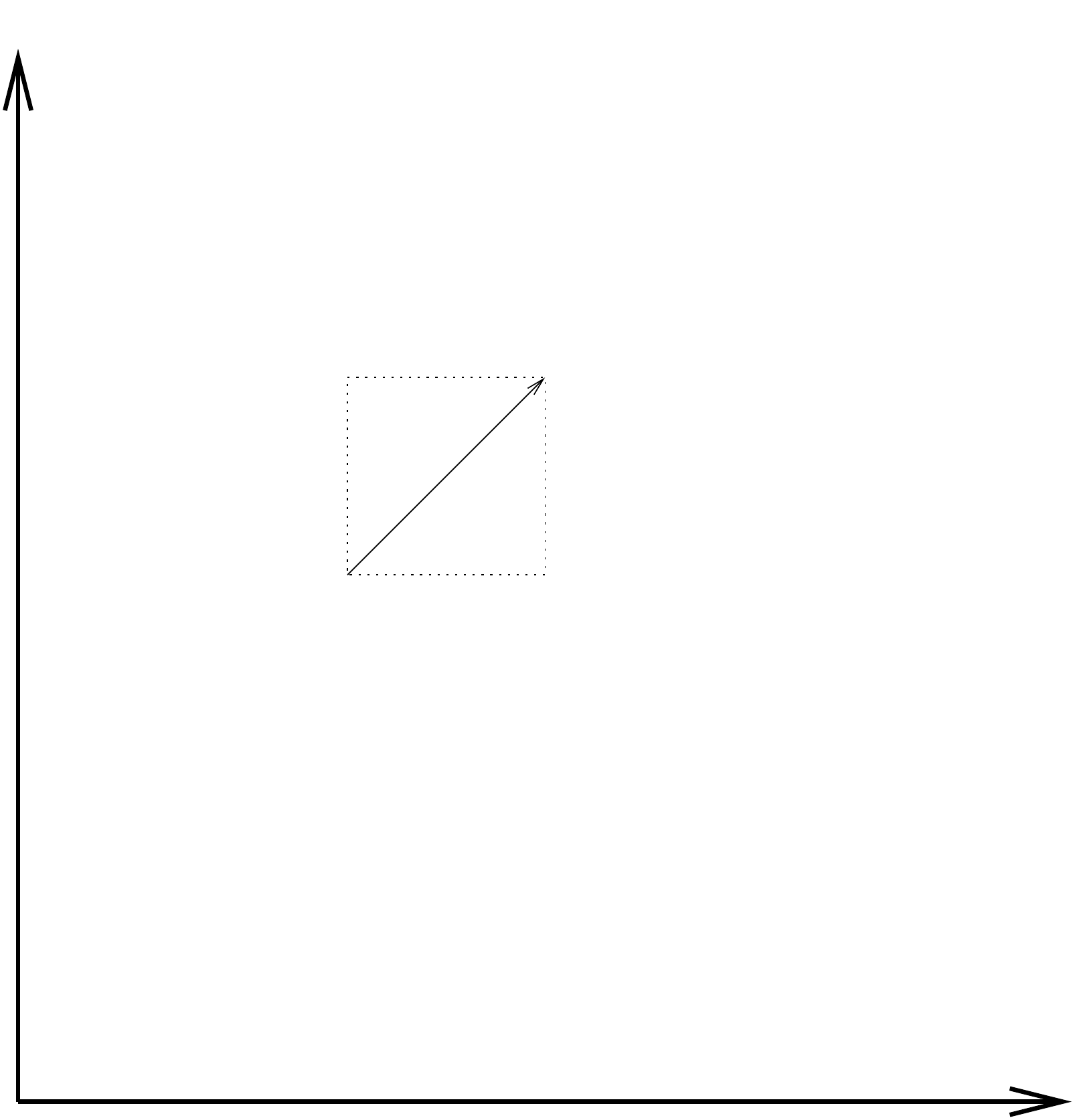}
\caption{$i\in I_0$ or $j\in I_0$: trajectories never stay at the junction point}
\label{fig2}
\end{center}
\end{figure}
\begin{figure}[h]
\begin{center}
\raisebox{-4pt}{\resizebox{4.1cm}{!}{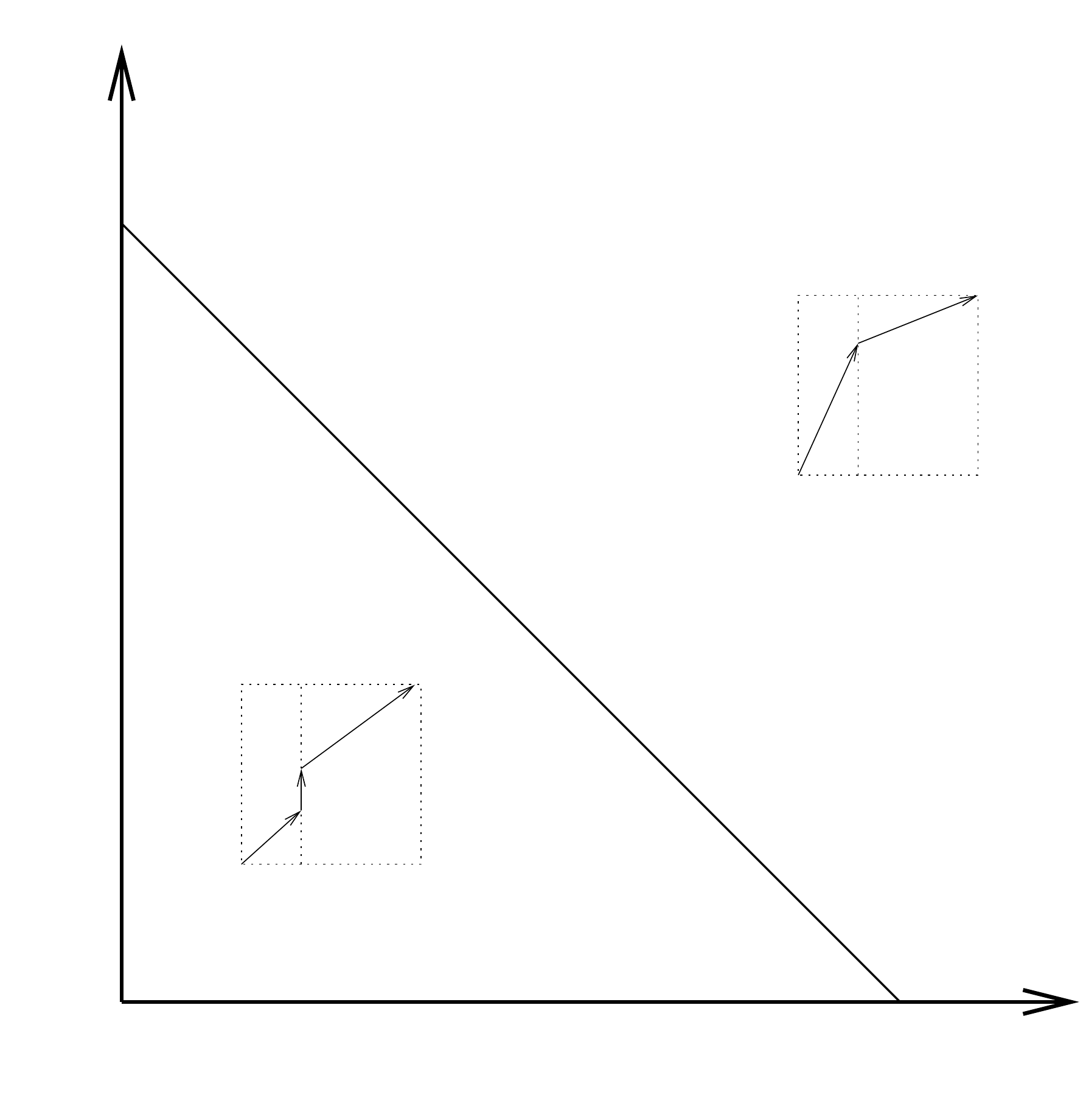}}
\resizebox{4cm}{!}{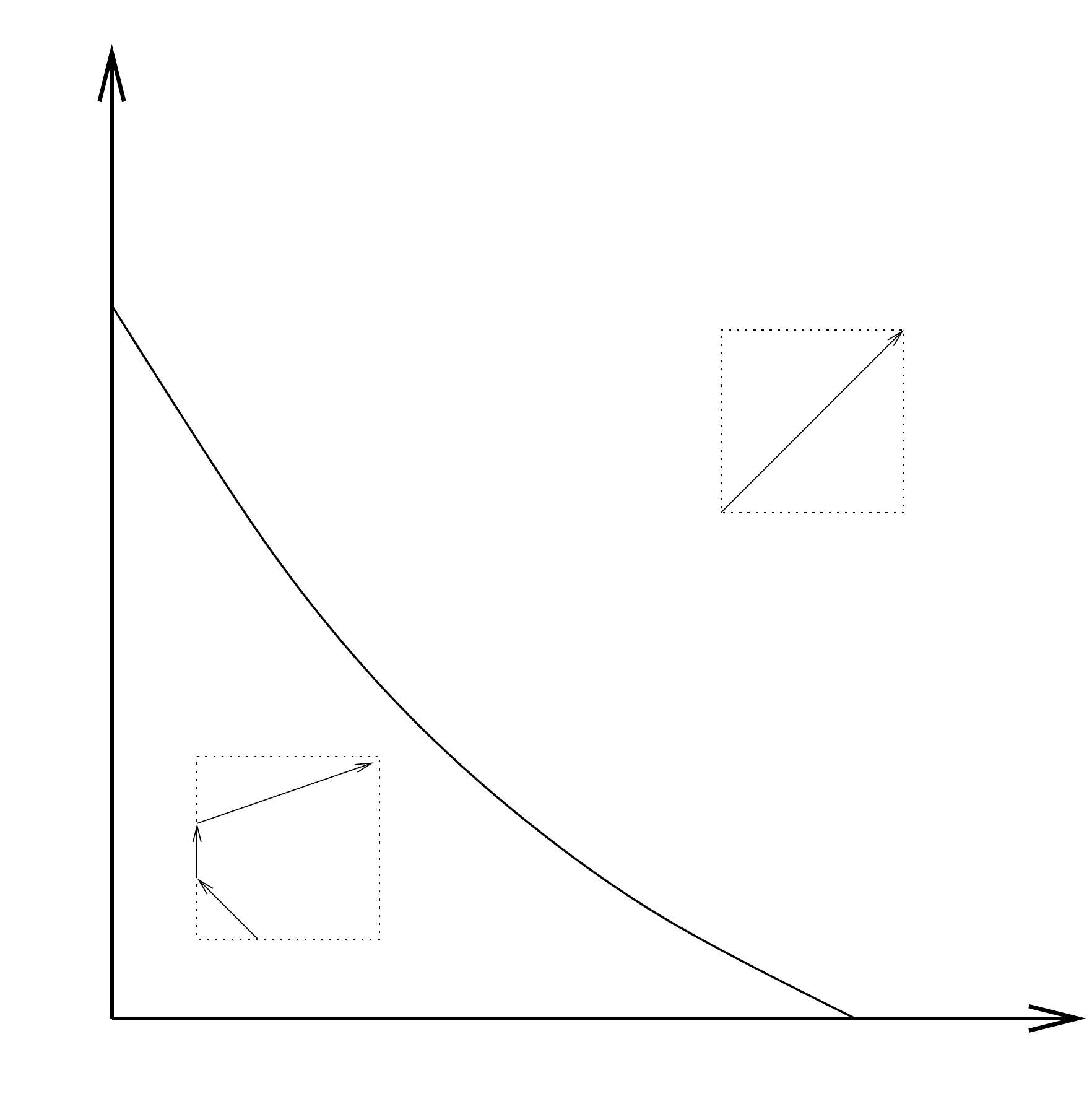}
\caption{$i,j \in I_N \setminus I_0$: trajectories do stay at the junction
  point if $(y,x) \sim (0,0)$}
\label{fig4}
\end{center}
\end{figure} 
\begin{proof}[Proof of Lemma~\ref{lem:dzo-1}.]
  We only have to treat the case $i=j$. The convexity of $L_i$ implies that for
  $\tau\in (0,1)$:
$$
e(\tau,y,x)= \tau L_i\left(-\frac{y}{\tau}\right)+
(1-\tau)L_i\left(\frac{x}{1-\tau}\right)\ge L_i(x-y).
$$
Therefore for $(y,x)\in J_j^*\times J_i^*$ with $j=i$, we have
$$
\dzt^{ji}(y,x) = \inf_{0< \tau< 1}e(\tau,y,x) \ge L_i(x-y)
$$
When $i=j\in I_0$,  we have
$\dzun^{ji}(y,x) =\dzt^{ji}(y,x)$, and then
$$
\dzun^{ji}(y,x)   \ge L_i(x-y) = \dzq^{ji}(y,x)=\dzo^{ji}(y,x)
$$
for $(y,x)\in J_j^*\times J_i^*$ and then also for $(y,x)\in J_j\times
J_i$, by continuity of the functions.
\end{proof}
\begin{proof}[Proof of Lemma~\ref{lem:dzo-1bis}.]
We first prove \eqref{eq::r120} and then describe the curve
$\Gamma_{j,i}$. 

\paragraph{Proof of (\ref{eq::r120}).}
Combining \eqref{eq:dzo-dzun} and \eqref{eq:dzun-dzd-dzt}, we obtain
\begin{multline}\label{eq::r121}
  \dzo^{ji}(y,x)= \min
  (\dzun^{ji}(y,x),\dzq^{ji}(y,x))\\=\begin{cases}
      \dzq^{ji}(y,x)=L_i(x-y) &  \text { for } (y,x)\in (J_j\times J_i)\setminus \Delta^{ji},\\
      \min(\dzd^{ji}(y,x),\dzq^{ji}(y,x)) &  \text{ for }      (y,x)\in \Delta^{ji}.
\end{cases}
\end{multline}
On the other hand, we have (a strictly convex function being above its
tangent) for $x\not= \xi_i^+$ and $y\not= -\xi_j^-$
\begin{align*}
L_i(x)&> L_i(\xi_i^+) +
    (x-\xi_i^+)L_i'(\xi_i^+) = x L_i'(\xi_i^+) + L_0(0)
    = \dzd^{ji}(0,x) \\
    L_j(-y) &> L_j(\xi_j^-) + (-y- \xi_j^-)L_j'(\xi_j^-) =
    -yL_j'(\xi_j^-) +L_0(0) = \dzd^{ji}(y,0). 
\end{align*}
This shows that
\begin{equation}\label{eq::r122}
  \dzq^{ji} > \dzd^{ji} \quad \mbox{on}\quad  
  (\partial (J_j\times J_i)) \cap \Delta^{ji}.
\end{equation}
We see that (\ref{eq::r121}) and (\ref{eq::r122}) imply
(\ref{eq::r120}).
\medskip

\paragraph{Description of $\left\{\dzd^{ji}=\dzq^{ji}\right\}\cap
  \Delta^{ji}$.}
Notice that
$$
\left\{\begin{array}{l}
    \dzd^{ji}(0,\xi_i^+)=\xi_i^+ L_i'(\xi_i^+)+L_0(0) = L_i(\xi_i^+)=\dzq^{ji}(0,\xi_i^+),\\
    \dzd^{ji}(-\xi_j^-,0)=\xi_j^- L_j'(\xi_j^-)+L_0(0) = L_j(\xi_j^-)=\dzq^{ji}(-\xi_j^-,0).\\
\end{array}\right.
$$
This means that the functions $\dzd^{ji}$ and $\dzq^{ji}$ coincide at
the two points $X_i=(0,\xi_i^+)$ and $Y_j=(-\xi_j^-,0)$.  Therefore we
have
$$
\dzq^{ji}< \dzd^{ji}\quad \mbox{on the open interval}\quad ]X_i,Y_j[
$$
because $\dzd^{ji}$ is linear and $\dzq^{ji}$ is strictly convex as a function of $y-x$.\\
The function $(y,x)\mapsto \dzq^{ji}(y,x)-\dzd^{ji}(y,x)$ being convex
because $\dzd^{ji}(y,x)$ is linear, we can consider the convex set
$$
K^{ji}=\left\{(y,x)\in J_j\times J_i,\quad \dzq^{ji}(y,x)\le
  \dzd^{ji}(y,x)\right\}.
$$
Then for $i=j\in I_N\setminus I_0$, the set
$$
\Gamma^{ji}=\left\{(y,x)\in \Delta^{ji},\quad
  \dzd^{ji}(y,x)=\dzq^{ji}(y,x)\right\}
$$
is contained in the boundary of the convex set $K^{ji}$. More precisely, we have
$$
\Gamma^{ji} = ((\partial K^{ji}) \cap \Delta^{ji})\subset J_j^*\times
J_i^*
$$
which shows that $\Gamma^{ji}$ is a curve and
$$
\overline{\Gamma^{ji}}=\Gamma^{ji} \cup \left\{X_i,Y_j\right\}.
$$
\end{proof}
\begin{lemma}[The equations for $\dzo$] \label{lem:dzo-2}
For all $i,j$, and $x,y$ where $\dzo^{ji}$ is $C^1$:
\begin{multline}\label{eq::r133}
  (\dzo^{ji}-x\partial_x\dzo^{ji}-y\partial_y\dzo^{ji})(y,x) \\
  = -H_i((\partial_x \dzo^{ji})(y,x)) = -H_j((-\partial_y
  \dzo^{ji})(y,x)).
\end{multline}
Moreover  for all $x\in J_i$ (with $x\not=\xi_i^+$ if $j=i\in I_N\setminus I_0$)
\begin{equation}\label{eq::r131}
(\dzo^{ji}-x\partial_x\dzo^{ji}-y\partial_y\dzo^{ji})(0,x) = L_0(0)+K_i(\max(x,\xi_i^+))
\end{equation}
and for all $y\in J_j$ (with $y\not=-\xi_j^-$ if $j=i\in I_N\setminus I_0$)
\begin{equation}\label{eq::r132}
(\dzo^{ji}-x\partial_x\dzo^{ji}-y\partial_y\dzo^{ji})(y,0)=L_0(0)+K_j(-\max(y,-\xi_j^-)).
\end{equation}
We also have
\begin{equation}\label{eq::r134}
  \partial_x \dzo^{ji}(y,x)= L_i'(\xi_x),\quad \partial_y \dzo^{ji}(y,x)= -L_j'(\xi_y)
\end{equation}
for all $(y,x)\in \partial (J_j\times J_i)$ except for $i=j\in
I_N\setminus I_0$ for which we exclude points $(y,x)\in
\left\{Y_j,X_i\right\}$.

Moreover for $j=i\in I_0$, we have
\begin{equation}\label{eq::r140}
\xi_x=\xi_y = x-y \quad \mbox{for all}\quad (y,x)\in \partial (J_j\times J_i)
\end{equation}
and $j=i\in I_N\setminus I_0$, we have
\begin{equation}\label{eq::r141}
\left\{
\begin{array}{lll}
       \xi_y=x-y,&\quad \xi_x=x-y &    \\
       \xi_y=\xi_j^-,& \quad \xi_x=\xi_i^+ & \quad \mbox{ for }\quad
       (y,x)\in 
       (\partial(J_j\times J_i)) \cap \overline{\Delta^{ji}}
 \end{array}\right.
\end{equation}
\end{lemma}
\begin{proof}
  Using Proposition~\ref{lemma:junc-1} for $\dzun^{ji}$,
  Lemma~\ref{lem::r123} for $\dzq^{ji}$, and (\ref{eq::r130}) for
  $\dzd^{ji}$ and the property (\ref{eq::r120}), we get
$$
(\dzo^{ji}-x\partial_x\dzo^{ji}-y\partial_y\dzo^{ji})(0,x)
= 
\left\{\begin{array}{ll}
L_0(0)+K_i(\max(x,\xi_i^+)) & \quad \mbox{if}\quad i\not=j\\
L_0(0)+K_i(x)  & \quad \mbox{if}\quad i=j\in I_0\\
\left\{\begin{array}{ll}
L_0(0)+K_i(x)   &\quad \mbox{if}\quad x>\xi_i^+\\
L_0(0)  &\quad \mbox{if}\quad x<\xi_i^+
\end{array}\right| & \quad \mbox{if}\quad i=j\in I_N\setminus I_0
\end{array}\right.
$$
which implies (\ref{eq::r131}).  Similarly we get
$$
(\dzo^{ji}-x\partial_x\dzo^{ji}-y\partial_y\dzo^{ji})(y,0)
= 
\left\{\begin{array}{ll}
L_0(0)+K_j(-\max(y,-\xi_j^-))  & \quad \mbox{if}\quad i\not=j\\
L_0(0)+K_j(-y)  & \quad \mbox{if}\quad i=j\in I_0\\
\left\{\begin{array}{ll}
L_0(0)+K_j(-y) &\quad \mbox{if}\quad y>-\xi_j^-\\
L_0(0) &\quad \mbox{if}\quad y<-\xi_j^-
\end{array}\right|& \quad \mbox{if}\quad i=j\in I_N\setminus I_0
\end{array}\right.
$$
which implies (\ref{eq::r132}).  Relations (\ref{eq::r133}) and
(\ref{eq::r134}) follow both from Proposition~\ref{lemma:junc-1} and Lemma
\ref{lem::r123}.  Finally (\ref{eq::r140}) and (\ref{eq::r141})
follows from the previous results.  This ends the proof of the lemma.
\end{proof}
We now can check the equations satisfied by $\dzo$ at the boundary.
\begin{lemma}[Boundary properties of $\dzo$]\label{lem:dzo-bdry-1}
  At any point $(y,x)\in \left\{0\right\}\times J_i$ with $x\not=
  \xi_i^+$ if $i\in I_N\setminus I_0$, we have for any $j\in I_N$
\begin{equation}\label{eq::r136}
  (\dzo^{ji}-x\partial_x\dzo^{ji}-y\partial_y\dzo^{ji})(y,x) 
  = \left\{\begin{array}{rl}
      -\max_{k\in I_N} H_k^-(-\partial_y \dzo^{ki}(y,x)) & \quad \mbox{if}\quad N\ge 2,\\
      - H_1(-\partial_y \dzo^{ki}(y,x)) & \quad \mbox{if}\quad N=1.
\end{array}\right.
\end{equation}
\end{lemma}
\begin{lemma}[Boundary properties of $\dzo$ (continued)]\label{lem:dzo-bdry-2}
At any point $(y,x)\in J_j\times \left\{0\right\}$ with
$y\not= -\xi_j^-$ if $j\in I_N\setminus I_0$, we have for any $i\in
I_N$
\begin{equation}\label{eq::r137}
  (\dzo^{ji}-x\partial_x\dzo^{ji}-y\partial_y\dzo^{ji})(y,x) 
=-\max_{k\in I_N} H_k^-(\partial_x \dzo^{jk}(y,x)).
\end{equation}
\end{lemma}
\begin{proof}[Proof of Lemma~\ref{lem:dzo-bdry-1}.]
We first remark the general fact that 
$$
H_k(L_k'(\xi))=H_k^-(L_k'(\xi)) \quad \mbox{if}\quad \xi\le 0.
$$
On the one hand, from Lemma~\ref{lem:dzo-1}, we have for points
$(y,x)\in \left\{0\right\}\times J_i$ where $\dzo^{ki}$ is $C^1$
$$
-(\dzo^{ki}-x\partial_x\dzo^{ki}-y\partial_y\dzo^{ki})(y,x)
=H_k(-\partial_y \dzo^{ki}(y,x)) \ge H_k^-(-\partial_y \dzo^{ki}(y,x))
$$ 
and this common quantity is independent on $k$.  Therefore to conclude
to (\ref{eq::r136}) in the case $N\ge 2$, it is enough to show that
there exists at least an index $k$ such that
\begin{equation}\label{eq::r137bis}
  H_k(-\partial_y \dzo^{ki}(y,x)) =  H_k^-(-\partial_y \dzo^{ki}(y,x)).
\end{equation}

\paragraph{Case 1:} $N\ge 2$ and $k\not= i$.
Then we have $\xi_y\le 0$ and then
$$
H_k(-\partial_y \dzo^{ki}(y,x)) =
H_k(L_k'(\xi_y))=H_k^-(L_k'(\xi_y))=H_k^-(-\partial_y \dzo^{ki}(y,x)).
$$
Therefore \eqref{eq::r137bis} holds true for $k\not=i$.

\paragraph{Case 2:} $N=1$ and $k= i=1\in I_0$.
Then we have
$$\dzo(y,x)=\dzo^{11}(y,x)=L_1(x-y)$$
and by Lemma~\ref{lem::r123}, we have for
$$
(\dzo^{11}-x\partial_x\dzo^{11}-y\partial_y\dzo^{11})(y,x)=
-H_1(-\partial_y \dzo^{11}(y,x))
$$ 
which is in particular true for $y=0$. This shows (\ref{eq::r136}) in
the case $N=1$.  \medskip
\end{proof}
\begin{proof}[Proof of Lemma~\ref{lem:dzo-bdry-2}.]
From Lemma~\ref{lem:dzo-1}, 
we have for points $(y,x)\in J_j\times \left\{0\right\}$ where $\dzo^{jk}$ is $C^1$
$$
-(\dzo^{jk}-x\partial_x\dzo^{jk}-y\partial_y\dzo^{jk})(y,x) 
= H_k(\partial_x \dzo^{jk}(y,x)) \ge H_k^-(\partial_x
\dzo^{jk}(y,x))
$$ 
and this common quantity is independent on $k$. Therefore to conclude
to (\ref{eq::r137}), it is enough to show that there exists at least
an index $k$ such that
\begin{equation}\label{eq::r137ter}
  H_k(\partial_x \dzo^{jk}(y,x)) =  H_k^-(\partial_x \dzo^{jk}(y,x)).
\end{equation}

\noindent \textbf{Case 1:} $j\in I_0$.
Then from Lemma~\ref{lem:dzo-1}, we have with $k=j$
\begin{equation}\label{eq::r142}
\partial_x \dzo^{jk}(y,x)=L_k'(\xi_x) \quad \mbox{with}\quad \xi_x=x-y\le 0.
\end{equation}
Therefore (\ref{eq::r137}) holds true for $k=j$.
\medskip

\noindent \textbf{Case 2:} $j\in I_N\setminus I_0$.
We distinguish sub-cases. \medskip

\noindent \textsc{Subcase 2.1:} $y> -\xi_j^-$. From Lemma~\ref{lem:dzo-1}, we
still have (\ref{eq::r142}) with $k=j$, which again implies
(\ref{eq::r137}) for $k=j$.\medskip

\noindent \textsc{Subcase 2.2:} $y< -\xi_j^-$. Then we choose an index $k\in
I_0$, and Lemma~\ref{lem:dzo-1} implies that
$$
\partial_x \dzo^{jk}(y,x)=L_k'(\xi_x) \quad \mbox{with}\quad
\xi_x=\xi_k^+ =0
$$
which again implies (\ref{eq::r137}) for such $k\in I_0$.  This ends
the proof of the lemma.
\end{proof}
We can now prove Theorem~\ref{thm:minimal-action-0}.
\begin{proof}[Proof of Theorem~\ref{thm:minimal-action-0}.]
  From Lemma~\ref{lem:dzo-1}, we know that $\dzo$ has the regularity
  $C^1_*$ except on certain curves $\Gamma^{ji}\cup
  \left\{Y_j,X_i\right\}$ for $j=i\in I_N\setminus I_0$. So if
  $(y,x)$ is a point of local $C^1_*$ regularity of $\dzo$, then we
  simply set
$$
\varphi_0 =  \dzo \quad \mbox{locally around}\quad (y,x).
$$
If $(y,x)$ is a point where $\dzo$ is not $C^1_*$, then we have
$\dzo(y,x)=\dzun(y,x)$, and we can simply set
$$
\varphi_0 = \dzun \quad \mbox{on}\quad J^2.
$$
The required equalities follow from Lemmas~\ref{lem:dzo-1},
\ref{lem:dzo-bdry-1}, \ref{lem:dzo-bdry-2}.
Estimate~(\ref{eq:grad-control}) follows from the fact that $\dzo$ is
the minimum of $\dzun \in C^1_*(J^2)$ and of functions in $C^1(J_i^2)$
for some $i$.  This ends the proof of the theorem.
\end{proof}

\appendix

\section{Appendix: Stability and Perron's method}

This section contains classical results from viscosity solutions, whose
statements are adapted to the equation studied in the present paper. 

\subsection{Stability results}

In view of Proposition~\ref{prop:equiv-def}, the following stability
results are classical in the viscosity solution framework. See for instance
\cite{barles94}.
\begin{proposition}[Stability]\label{prop:stab}
Assume (A1') and let $T>0$.
\begin{itemize}
\item
Consider a family of sub solutions (resp. super-solutions) $(u_\alpha)_{\alpha \in A}$ 
of \eqref{eq:main} on $J_T$ such that the u.s.c. (resp. l.s.c.) envelope $u$ of 
$$
\sup_{\alpha \in A} u_\alpha \quad \quad (\mbox{resp.} \ \inf_{\alpha \in A} u_\alpha)
$$ 
is finite everywhere. Then $u$ is a sub-solution (resp. super-solution)
of \eqref{eq:main} on $J_T$.
\item Consider a family of sub solutions (resp. super-solutions)
  $(u_\eps)_{\eps \in (0,1)}$ of \eqref{eq:main} on $J_T$ such that
  the upper (resp. lower) relaxed semi-limit $u$ is finite
  everywhere. Then $u$ is a sub-solution (resp. super-solution) of
  \eqref{eq:main} on $J_T$.
\end{itemize}
\end{proposition}

\subsection{Perron's method}

In this subsection, we state the existence of a solution of
\eqref{eq:main}-\eqref{eq:ci} which can be constructed by using
Perron's method. This method is the classical way to get existence in
a viscosity solution framework.  
\begin{theorem}[Existence]\label{thm:perron}
  Assume (A0)-(A1') and let $T>0$.  Then there exists an upper
  semi-continuous function $u\colon [0,T)\times J \to \R$ which is a
  viscosity solution of \eqref{eq:main}-\eqref{eq:ci} on $J_T$ and
  satisfies
$$|u(t,x)-u_0(x)|\le Ct \quad \mbox{for}\quad (t,x)\in [0,T)\times J$$
\end{theorem}

\paragraph{Acknowledgements.} The first author is partially supported
by the KAMFAIBLE project (ANR-07-BLAN-0361). The authors thank
A. Briani for useful discussions and important indications on the
recent literature on the subject. The authors are very grateful to
J.-P. Lebacque for fruitful and enlightening discussions on the
modeling of traffic flows and indications on the related
literature. The second author thanks M.  Khoshyaran for discussions on
the topic.  The authors also thank Y. Achdou, G. Barles,
P. Cardaliaguet, E. Chasseigne, H. Frankowska and M.~Garavello for
discussions about this work.

\small

\def\polhk#1{\setbox0=\hbox{#1}{\ooalign{\hidewidth
  \lower1.5ex\hbox{`}\hidewidth\crcr\unhbox0}}}

\end{document}